\documentclass[12pt, reqno]{amsart}

\usepackage{comment}
\usepackage{amssymb,amsmath,amsthm,wasysym}
\usepackage{graphicx}
\usepackage{paralist}
\usepackage{stackrel}
\usepackage{tikz-cd}
\usepackage{tikz}
\usepackage[all,cmtip]{xy}
\usetikzlibrary{arrows}
\usepackage{tcolorbox}

\makeatletter
\renewcommand*\env@matrix[1][*\c@MaxMatrixCols c]{%
  \hskip -\arraycolsep
  \let\@ifnextchar\new@ifnextchar
  \array{#1}}
\makeatother

\def\cP{\mathcal P}

\def\cB{\mathcal B}

\newcommand{\Q}{\mathbb{Q}}

\newcommand{\R}{{\mathbb{R}}}
\newcommand{\F}{\mathbb{F}}

\newcommand{\C}{\mathbb{C}}
\newcommand{\Z}{\mathbb{Z}}
\newcommand{\N}{\mathbb{N}}

\newcommand{\fa}{{\mathfrak a}}

\newcommand{\fm}{{\mathfrak m}}

\newcommand{\fg}{{\mathfrak g}}
\newcommand{\gl}{\mathfrak {gl}}

\def\Ann{\operatorname{Ann}}

\def\NLL{\operatorname{NLL}}
\def\End{\operatorname{End}}

\def\Hess{\operatorname{Hess}}
\def\hess{\operatorname{hess}}

\def\im{\operatorname{Im}}
\def\ker{\operatorname{Ker}}

\def\chr{\operatorname{char}}

\def\chr{\operatorname{char}}

\def\tr{\operatorname{tr}}

\def\weight{\operatorname{weight}}

\def\lra{\longrightarrow}
\def\ra{\rightarrow}

\newcommand\mapsfrom{\mathrel{\reflectbox{\ensuremath{\mapsto}}}}

\def\sl{\mathfrak{sl}}

\newcommand{\contract}{\bullet}

\usepackage[colorlinks=true, linkcolor=blue, anchorcolor=black, citecolor=blue, filecolor=blue, menucolor= blue, urlcolor=blue]{hyperref}
\usepackage[nameinlink, capitalize]{cleveref}

\numberwithin{equation}{section}

\theoremstyle{plain} 
\newtheorem{thm}{Theorem}[section]

\newtheorem*{thm*}{Theorem}
\newtheorem{question}[thm]{Question}
\newtheorem*{question*}{Question}
\newtheorem{cor}[thm]{Corollary}
\newtheorem*{cor*}{Corollary}
\newtheorem{lem}[thm]{Lemma}
\newtheorem{prop}[thm]{Proposition}

\theoremstyle{definition}
\newtheorem{defn}[thm]{Definition}
\newtheorem*{defn*}{Definition}

\newtheorem{construction}[thm]{Construction}

\newtheorem{ex}[thm]{Example}
\newtheorem{exer}[thm]{Exercise}
\newtheorem*{exer*}{Exercise}

\newtheorem*{notation*}{Notation}

\theoremstyle{remark}
\newtheorem{rem}[thm]{Remark}
\newtheorem*{rem*}{Remark}

\def\C{{\mathbb{C}}}
\def\F{{\mathbb{F}}}
\def\P{{\mathbb{P}}}
\def\R{{\mathbb{R}}}

\def\cB{{\mathcal{B}}}

\DeclareMathOperator{\Hom}{Hom}
\DeclareMathOperator{\Span}{Span}

\def\ra{\rightarrow}
\def\la{\leftarrow}

\usepackage[top=1.2in, bottom=1.2in, left=1.2in, right=1.2in]{geometry}

\makeatletter
\newcommand*{\toccontents}{\@starttoc{toc}}
\makeatother

\begin{document}

\title{Lefschetz properties through a topological lens}
\author{Alexandra Seceleanu}
\address{Mathematics Department, University of Nebraska--Lincoln, 203 Avery Hall}
\email{aseceleanu@unl.edu}
\date{}
\thanks{Partially supported by NSF DMS--2401482.}

\maketitle 


\begin{abstract}
These notes were prepared for the {\em Lefschetz Preparatory School}, a graduate summer course held in Krakow, May 6--10, 2024. They present the story of the algebraic Lefschetz properties from their origin in algebraic geometry to some recent developments in commutative algebra. The common thread of the notes is a bias towards topics surrounding the algebraic Lefschetz properties that have a topological flavor. These range from the Hard Lefschetz Theorem for cohomology rings to commutative algebraic analogues of these rings, namely artinian Gorenstein rings, and topologically motivated operations among such rings. 
\end{abstract}

\tableofcontents

\pagenumbering{arabic}

\section{Introduction}

The topic of these notes is the algebraic Lefschetz properties, which are abstractions of  the important Hard Lefschetz Theorem from complex geometry. \cref{s: 1} explains the topological context of this result. \cref{s: 3} introduces the algebraic Lefschetz properties and their relevance to commutative algebra. \cref{s: 4} establishes a correspondence between the strong Lefschetz property and an action of the Lie group $\sl_2$.  \cref{s: 6} focuses on the class of Gorenstein rings and their construction via Macaulay's inverse system. \cref{s: 5}  and \cref{s: 7} investigate how various constructions of new rings from old interact with the Lefschetz properties.

These notes are deeply influenced by the monograph \cite{book} by T.~Harima, T.~Maeno, H.~Morita, Y.~Numata, A.~Wachi and J.~Watanabe. This reference contains a parallel description of many topics  in these notes except for \cref{s: 7}, which describes more recent developments based on \cite{IMMSW}. The treatment of earlier chapters, while deeply influenced by \cite{book}, reflects the author's mathematical taste.



\section{Cohomology rings and the Hard Lefschetz Theorem} \label{s: 1}
This section gives an introduction to the origins of the algebraic Lefschetz properties. The motivation for this topic comes from algebraic topology, so we will spend some time looking at how the Lefschetz property arises there.

\subsection{Cohomology rings} 
Let $\F$ be a vector space and let $X$ be a topological space (such as projective space $\P^n$ or the $n$-dimensional sphere $S^n$).  We recall the notion of cohomology of $X$ with coefficients in $\F$.

First, one can think of $X$ as being made out of simple cells (or at least one can approximate $X$ in this manner). This endows $X$ with a {\em cell complex (CW-complex)} structure. 

\begin{ex}[CW structures on sphere]
\label{ex:CW sphere}
The 2-dimensional sphere $S^2$ can be obtained by taking a point (0-dimensional cell) and glueing a 2-dimensional disc onto it along its entire boundary. So the CW-structure of $S^2$ is
$$S^2= \text{pt} + \text{2-dimensional disc}$$
More generally  one can do the same for the $n$-dimensional sphere $S^n$:
$$S^n= \text{pt} + n\text{-dimensional disc}.$$

There is another, less economical way to give the sphere a CW-structure. For $S^2$ one takes two $0$-dimensional cells, connects them using two line segments ($1$-dimensional cells) to form a circle $S^1$. Then one can glue two $2$-dimensional discs via their boundaries to the circle to form $S^2$.
Similarly, there is a CW-structure on $S^n$ with two cells in each dimension summarized by
$$S^n= 2\times \text{pt} + 2\times  \text{1-dimensional disc} +  2\times \text{2-dimensional disc} + \cdots +2\times \text{n-dimensional disc}.$$

\end{ex}

\begin{ex}[CW structure on the real projective space]
\label{ex: CW proj}
Consider first $\P_\R^n$. It can be written as $S^n/ \{\pm 1\}$. If we take a CW structure on $S^n$ with two cells in each dimension, then the action of $-1$ swaps the cells, thus they become identified in the quotient. Due to this $\P_\R^n$ has a CW structure with one cell in each dimension. 
$$\P_\R^n= \text{pt} + \text{1-dimensional cell}+ \dots +\text{n-dimensional cell}.$$

Next consider $\P_\C^n$. This has a cell in every even (real) dimension:
$$\P_\C^n= \text{pt} + \text{2-dimensional cell}+ \dots +\text{2n-dimensional cell}.$$
\end{ex}

Proceeding towards homology, we define a {\em chain complex} ${\bf C_\bullet}(X)$ by letting $C_n(X)$ be the  $\F$-vector space generated by the $n$-dimensional cells of $X$. There are so-called boundary maps\footnote{We will not describe the boundary maps here.}, which fit into the following sequence
$${\bf C_\bullet}(X):  0\la \F^{\# \text{0-cells}} \la \F^{\# \text{1-cells}} \la \cdots \la \F^{\# \dim(X)\text{-cells}}\la 0.$$
There is also a dual version  called the {\em cochain complex} of $X$ with coefficients in $R$
$${\bf C^\bullet}(X)=\Hom({\bf C_\bullet}(X),\F):  0\ra \F^{\# \text{0-cells}} \stackrel{\partial_1}{\ra} \F^{\# \text{1-cells}} \stackrel{\partial_2}{\ra} \cdots \stackrel{\partial_n}{\ra} \F^{\# \dim(X)\text{-cells}}\ra 0.$$

\begin{defn}
The {\em cohomology groups} of $X$ are defined as 
$$H^i(X,\F)=H^i\left({\bf C^\bullet}(X)\right)=\ker{\partial_i}/\im{\partial_{i-1}}.$$
\end{defn}

\begin{ex}\label{ex: H sphere}
Based on \cref{ex:CW sphere}  we have the following chain complexes, which lead to easy computations of the corresponding cohomology groups.
$${\bf C^\bullet}(S^n): 0\ra \F \ra 0 \ra 0 \ra \dots \ra \F\ra 0 $$
$$ H^i(S^n,\F)=\begin{cases} \F & i=0,n \\0 & \text{ otherwise} \end{cases}$$

$${\bf C^\bullet}(\P_\C^n): 0\ra \F\ra 0 \ra \F \ra 0 \ra \F \ra  \dots\ra  \F \ra 0 $$
$$ H^i(\P_\C^n,R)=\begin{cases} \F &  i=\text{even } \\ 0 & i=\text{odd.} \end{cases}$$
\end{ex}

The special property of these cohomology groups that allows us to study them using tools from ring theory is that they can be assembled into a graded ring.
\begin{defn}
The {\em cohomology ring} of $X$ is
$$H^\bullet(X,\F)=\bigoplus_{i\geq 0}H^i(X,\F).$$
 To study multiplication on this ring we need to define a map called the {\em cup product} 
 $$H^m(X,\F)\times H^n(X,\F)\to H^{m+n}(X,\F).$$ For this recall the K\"unneth isomorphism: for two topological spaces $X$ and $Y$ if one of $X$ or $Y$ has torsion-free homology (this holds when working over a field $\F$) and has finitely many cells in each dimension, there is an isomorphism  $\mu:H^\bullet(X\times Y,\F)\cong H^\bullet(X,\F)\otimes_\F H^\bullet(Y,\F)$. The composite with the diagonal map
$$ H^\bullet(X,\F)\otimes_\F  H^\bullet(X,\F)\stackrel{\cong}{\ra} H^\bullet(X\times X,\F)\stackrel{\Delta^*}{\ra}  H^\bullet(X,\F)$$ defines the cup product by $x \cup y = \Delta^*\mu(x \otimes y)$.
The cup product is not commutative, but it is what we call {\em graded commutative}. This means that the ring is graded so that 
\begin{quote}
if $x\in H^m(X,\F)$, $|x|=m$ denotes the {\em cohomological degree} of $x$,
\end{quote}
  and any elements $x,y$ in this ring satisfy
\begin{equation}
\label{eq:gradedcomm}
x \cup y = (-1)^{|x||y|}y \cup x.
\end{equation}
Note that in a graded commutative ring even degree elements commute with all other elements, while odd degree elements anti-commute with other odd degree elements.
\end{defn}

\begin{ex}[Cohomology ring of a sphere]
From \cref{ex: H sphere} we have 
\[
H^\bullet(S^n,\F)=\F\oplus \F.
\]
Set $1$ and $e$ to be the basis of $H^0(S^n,\F)$ and $H^n(S^n,\F)$ as $\F$-vector spaces, respectively. Then $1$ is the multiplicative identity of the ring $H^\bullet(S^n,\F)$ and $e^2=e \cup e \in H^{2n}(S^n,\F) = 0$, so 
$$H^\bullet(S^n,\F)=\F[e]/(e^2) \text{ with } |e|=n.$$
\end{ex}

\begin{ex}[Cohomology ring of a torus]
\label{ex:torus}
Applying the K\"unneth formula to the torus $T^n=S^1\times \cdots \times S^1$ gives for elements $e_1,\ldots, e_n$ with $|e_i|=1$ 
\[
H^\bullet(T^n,\F)=\F[e_1]/(e_1^2)\otimes_\F \F[e_2]/(e_2^2)\otimes_\F \F[e_n]/(e_n^2)=\bigwedge_\F\langle e_1,\ldots, e_n\rangle.
\]
Note that the tensor product above is taken in the category of graded-commutative algebras which implies that $e_ie_j=-e_je_i$ as expected since $|e_i|=|e_j|=1$. If the characteristic of $\F$ is not equal to 2 then this implies $e_i^2=0$ for all $i$. The ring above, denoted  $\bigwedge_\F\langle e_1,\ldots, e_n\rangle$, is called an {\em exterior algebra}. As an $\F$-vector space, a basis of the exterior algebra is given by all the square-free monomials in the variables $e_1, \ldots, e_n$. 
\end{ex}

\begin{ex}[Cohomology ring of projective space]\label{ex: H proj space}
From \cref{ex: H sphere} we have $H^\bullet(\P_\C^n,\F)=\F\oplus \F \oplus \cdots \oplus \F$, with $n$ summands in degrees $0,2, \ldots, 2n$. Set $x$ to be the generator of $H^2(\P_\C^n,\F)$. It turns out then that $x^i\neq 0\in H^{2i}(\P_\C^n,\F)$ for $1\leq i\leq n$, so $x^i$ generates $H^{2i}(\P_\C^n,\F)$. Moreover $x^{n+1}=0$ since $H^{2n+2}(\P_\C^n,\F)=0$. Thus we have $$H^\bullet(\P_\C^n,\F)=\F[x]/(x^{n+1}), \text{ with }|x|=2.$$
We can apply the K\"unneth formula to compute
\begin{eqnarray*}
H^\bullet(\P_\C^{d_1}\times \P_\C^{d_2}\times  \cdots \times\P_\C^{d_n},\F) &\cong& \F[x_1]/(x^{d_1+1})\otimes_\F \F[x_2]/(x^{d_2+1})\otimes_\F \cdots \otimes_\F \F[x_n]/(x^{d_n+1})\\
&\cong& \F[x_1,\ldots,x_n]/(x_1^{d_1+1}, \ldots, x_n^{d_n+1}), \text{ with }|x_i|=2.
 \end{eqnarray*}
\end{ex}

\subsection{The Hard Lefschetz Theorem}

Solomon Lefschetz (1888--1972)  was a prominent mathematician who did fundamental work on algebraic topology, its applications to algebraic geometry, and the theory of non-linear ordinary differential equations. His career, including transitions from industry to mathematics and from working in Nebraska and Kansas to Princeton University is beutifully summarized in his own words in \cite{Lefschetzbio}. He was also a great supporter of mathematics in developing countries, helping train a great number of Mexican mathematicians.
Lefschetz understood that cohomology rings can be used to study questions in algebraic geometry. Speaking about his work Lefschetz states:
\begin{quote}
{\em ``The harpoon of algebraic topology was  planted  in the body of the whale of algebraic geometry."}
\end{quote}

We now come to the main result that we have been building up to. Let $X$ be an algebraic subvariety of $P_\C^n$  and let $H$ denote a general hyperplane in $P_\C^n$. Then $X \cap H$ is a subvariety of $X$ of real codimension two and thus, by a, standard construction in algebraic geometry, represents a cohomology class $L \in H^2(X,\R)$ called  the class of a hyperplane section. In more detail, $L$ is an $\F$-linear homomorphism that takes a dimension 2 subvariety of $X$ (or a 2-dimensional cell of $X$ if we view this as a CW-complex), intersects it with $H$ and returns the number of points of intersection. This function is then extended $\F$-linearly.

\begin{thm}[Hard Lefschetz Theorem]
\label{HLT}
 Let $X$ be a smooth irreducible complex projective variety of complex dimension $n$ (real dimension $2n$), $H^\bullet(X)=H^\bullet(X,\R)$, and let $L \in H^2(X,\R)$ be the class of a  general hyperplane section. Then for $0\leq i\leq n$ the following maps are isomorphisms
 $$L^{i}:H^{n-i}(X)\to H^{n+i}(X), \text{ where } L^i(x)=\underbrace{L\cup \cdots \cup L}_{L^i} \cup \, x.$$
\end{thm}

\begin{rem}
The Hard Lefschetz theorem holds for $H^\bullet(X,\F)$ where $\F$ is any field of characteristic zero (not just $\R$), but the conclusion of the theorem is false in positive characteristic.
\end{rem}

The theorem above was first stated by Lefschetz in \cite{Lefschetz}, but his proof was not entirely rigorous.The first complete proof of \cref{HLT} was given by Hodge \cite{Hodge}. The standard proof given nowadays uses the representation theory of the Lie algebra $\sl_2(\C)$ and is due to Chern  \cite{Chern}.  Lefschetz's original proof was only recently made rigorous by Deligne\cite{Deligne}, who extended it to positive characteristic.

\begin{ex}[The Hard Lefschetz theorem in action]
See \Cref{ex: H proj space} for context. For $H^\bullet(P_\C^n)=\F[x]/(x^{n+1})$ the class of a hyperplane is $L=x$ (recall that $|x|=2$)  and it gives whenever $i\equiv n\pmod{2}$ isomorphisms
\begin{eqnarray*}
 H^{n-i}(P_\C^n)=x^{\frac{n-i}{2}}\F\xrightarrow{\times x^i} H^{n+i}(P_\C^n)= x^{\frac{n+i}{2}}\F \\
 x^{\frac{n-i}{2}}y\mapsto x^i(x^{\frac{n-i}{2}}y)=x^{\frac{n+i}{2}}y.
\end{eqnarray*}
\end{ex}

Cohomology rings of  $n$-dimensional complex projective varieties $X$ with coefficients in a field $\F$ satisfy the following properties:
\begin{enumerate}
\item[(1)] $H^\bullet(X,\F)$ is a graded commutative ring\footnote{It is worth cautioning that the phrase graded commutative ring  has a very different meaning from commutative graded ring. In the former odd degree elements anti-commute while in the latter all elements commute regardless of their degree.} in the sense of \eqref{eq:gradedcomm}. Its even part $A:=H^{2\bullet(}X,\F)=\bigoplus_{i\geq 0} H^{2i}(X,\F)$ is a commutative graded ring as defined in the next chapter. We can re-grade this ring by setting $|x|=i$ if $x\in H^{2i}(X,\F)$. With this convention we have $|L|=1$.
\item[(2)] $H^\bullet(X,\F)$ and $A$ are finite dimensional $\F$-vector spaces (so $A$ is an artinian ring cf.~\cref{def: artinian})
\item[(3)] $H^\bullet(X,\F)$ and $A$ satisfiy Poincar\'e duality (hence $A$ is a Gorenstein ring cf.~\cref{prop: Poincare}).
\end{enumerate}

The main goal of these notes is to demonstrate how one may hope to extend the Hard Lefschetz theorem (and some weaker versions thereof) to arbitrary rings which may not necessarily be cohomology rings, but satisfy at least some of the properties above. Thus we are motivated by the following question. 

\begin{question} Which commutative graded rings $A$ that are artinian or both artinian and Gorenstein also satisfy the conclusion of  the Hard Lefschetz theorem?
\end{question}

\section{Classes of graded rings} \label{s: 2}

From now on all rings will be commutative unless specified otherwise.

\subsection{Artinian algebras}

\begin{defn}[Graded ring]
A commutative ring $A$ is an ($\N$--){\em graded ring} provided it decomposes as $$A=\bigoplus_{i\geq 0}A_i$$ with $A_i$  abelian groups such that $\forall i,j \in \N$ $A_iA_j\subseteq A_{i+j}$ ($a\in A_i, b\in A_j \Rightarrow ab\in A_{i+j}$).
\end{defn}

From now on we restrict to graded rings $A$ with $A_0=\F$ a field. I will refer to these as $\F$-{\em algebras}. Note that in particular such a ring $A$ and each of its homogeneous components $A_i$ are $\F$ vector spaces. 

An $\N$--graded ring $A$ has a unique homogeneous maximal ideal, namely $\fm=\bigoplus_{i\geq 1}A_i$.

\begin{ex}
$A=\F[x_1,\dots,x_n]$ is the fundamental example of a graded ring with $A_i=$ the set of homogeneous polynomials of degree $i$. Note that the degree of $x_i$ is allowed to be an arbitrary positive integer.
\end{ex}

\begin{exer}
Show that if $A$ is a commutative, Noetherian, graded $\F$-algebra  then  $\dim_\F A_i$ is finite for each $i$.
\end{exer}

\begin{defn}[Hilbert function]
The {\em Hilbert function} of a Noetherian graded $\F$-algebra $A$ is the function 
$$h_A:\N\to \N, h_A(i)=\dim_\F A_i.$$
The {\em Hilbert series} of $A$ is the power series $H_A(t)=\sum_{i\geq 0} h_A(i)t^i$.
\end{defn}

\begin{exer}
Prove that the Hilbert function of the polynomial ring $R=\F[x_1,\ldots, x_n]$ is given by 
\[
h_R(i)=\binom{n+i-1}{i}, \, \forall i\geq 0
\]
and the Hilbert series is 
\[
H_R(t)=\frac{1}{(1-t)^n}.
\]
\end{exer}

\begin{ex}
\label{ex: truncated poly}
The Hilbert function of the truncated polynomial ring $A=\frac{\F[x_1,\ldots, x_n]}{(x_1,\ldots, x_n)^d}$ is given by 
\[
h_A(i)=\begin{cases}
\binom{n+i-1}{i} & \text{ if } 0\leq i < d\\
0 & \text{ if }  i \geq  d.
\end{cases}
\]
Thus $H_A(t)=\sum_{i=0}^{d-1}\binom{n+i-1}{i}  t^i$.
\end{ex}

\begin{ex}
\label{ex:222}
Consider  a field $\F$ and let $A=\F[x,y,z]/(x^2,y^2,z^2)$. 
Clearly, $A$ is a finite dimensional $\F$-vector space with basis given by the monomials $\{1,x,y,z,xy,yz,xz,xyz\}$. We see that the elements of $A$ have only four possible degrees 0,1,2,3 and moreover
\begin{eqnarray*}
A_0&=& \Span_\F\{1\}\cong \F \Rightarrow h_A(0)=1\\
A_1&=& \Span_\F\{x,y,z\}\cong \F^3 \Rightarrow h_A(1)=3\\
A_2&=& \Span_\F\{xy,yz,xz\}\cong \F^3 \Rightarrow h_A(2)=3\\
A_3&=& \Span_\F\{xyz\}\cong \F \Rightarrow h_A(3)=1\\
A_i&=& 0, \forall i\geq 4  \Rightarrow h_A(i)=0, \forall i\geq 4
\end{eqnarray*}

Thus $H_A(t)=1+3t+3t^2+t^3$.  
\end{ex}

In \cref{ex: truncated poly} and \cref{ex:222} the Hilbert series was in fact a polynomial, equivalently the Hilbert function was eventually equal to zero. We now define a class of graded rings which satisfy this property. 

\begin{defn}[Artinian ring] \label{def: artinian}
A (local or) graded $\F$-algebra $A$ with (homogeneous) maximal ideal $\fm$ is {\em artinian}  if any of the following equivalent conditions holds.

\begin{itemize}
\item[(a)] $A$ is finite dimensional as a $\F$-vector space.
\item[(b)] $A$ has Krull dimension zero.
\item[(c)] $\mathfrak m^d = 0$  for some (hence all sufficiently large) integers $d \geq 1$. If $A$ is graded this can be restated as $A_d=0$ for sufficiently large integers $d\geq 1$.
\item[(d)] $A$ satisfies the descending chain condition on ideals.
\item[(e)] There exists a descending sequence of ideals
$$ A=\fa_0 \supseteq \fa_1 \supseteq \fa_2 \supseteq  \dots \supseteq  \fa_\ell=0  \text{ such that } \fa_{i-1}/\fa_i\cong\F.$$  Such a sequence of ideals is called a {\em composition series}.
\end{itemize}

Moreover, if $R=\F[x_1,\ldots, x_n]$ is a polynomial ring and $A=R/I$ for some homogeneous ideal $I$ of $R$ then the conditions above are also equivalent to
\begin{itemize}
\item[(f)] For each $0 \leq i \leq n$ there is some integer $p_i$ such that $x_i^{p_i} \in I$.

\item[(g)] If $\F$ is algebraically closed, another equivalent condition is that the vanishing locus of $I$ in projective space is empty.

\end{itemize}
\end{defn}

\subsection{Artinian Gorenstein rings and complete intersections}

\begin{defn}[Socle]\label{def: AG}
For a graded artinian $\F$ algebra the maximal integer $d$ such that $A_d \neq 0$ is called the {\em maximal socle degree} of $A$. The {\em socle} of $A$ is the ideal 
$$(0:_A\fm)=\{x\in A \mid xy=0, \forall y\in \fm\}.$$ 
There is always a containment $A_d\subseteq (0:_A\fm)$, where $d$ denotes the maximal socle degree of $A$. When the converse holds, namely  $ (0:_A\fm) \subseteq A_d$, then $A$ is called a {\em level algebra}. This condition means that the socle is concentrated in a single degree.
\end{defn}

\begin{defn}[Artinian Gorenstein ring]
A graded $\F$-algebra is {\em artinian Gorenstein} (AG) if its socle is a one dimensional $\F$-vector space.
\end{defn}

An equivalent characterization of AG algebras is given by the following proposition.

\begin{prop}[Poincar\'e duality]
\label{prop: Poincare}
A graded $\F$-algebra  $A$ of  maximal socle degree $d$ is AG if and only if for each nonzero element $a_{soc}$ of $A_d$ there exists an $\F$-vector space homomorphism $\int_A: A \to \F$ called an {\em orientation}, satisfying the following properties:
\begin{enumerate}
\item if $b\in A_i$ for some $i<d$ then $\int_A b=0$,
\item the orientation induces an isomorphism $A_d\cong \F$ such that $\int_A a_{soc}=1$,
\item for each $0\leq i\leq d$ and each element $a\in A_i$ there exists a unique element $b\in A_{d-i}$
so that $\int_A ab=1$.
\end{enumerate}
\end{prop}

In \cref{s: BUG} we will use the notation $a_{soc}$ implicitly to mean fixing the unique orientation on $A$ that satisfies $\int_A a_{soc}=1$.

\begin{ex}
Continuing with \cref{ex:222},  the socle is $(0:_A\fm)=\Span_{\F}\{xyz\}$, a 1-dimensional $\F$-vector space. This shows that $A$ is Gorenstein. Take the orientation on $A$ to be specified by $\int_A xyz=1$. We see that the $\F$-basis elements $\{1,x,y,z, xy, yz, xz, xyz\}$ of $A$ form pairs with respect to the given orientation in the following manner
\begin{eqnarray*}
\int_A 1\cdot xyz  &=& 1\\
\int_A x\cdot yz  &=& 1\\
\int_A y\cdot xz  &=& 1\\
\int_A z\cdot xy  &=& 1.
\end{eqnarray*}
\end{ex}

\begin{exer}
\label{exer: Gor}
Show that the following is an artinian Gorenstein ring 
\[
R=\frac{\F[x,y,z]}{(xy,xz,yz, x^2-y^2,x^2-z^2)}.
\]
\end{exer}

\begin{exer}
Prove that if $A$ is a graded AG algebra of maximum socle degree $d$ then $h_A(i)=h_A(d-i)$ for each $0\leq i \leq d$. This is usually stated by saying AG algebras have symmetric Hilbert function.
\end{exer}

\begin{defn}
A graded {\em artinian} $\F$-algebra is a {\em complete intersection} (CI) if $A=R/I$ where $R=\F[x_1,\ldots, x_n]$ and $I=(f_1, \ldots, f_n)$, that is, $I$ is a homogeneous ideal generated by as many elements as there are variables in $R$.
\end{defn}

\begin{ex} \label{ex: monomial CI}
The rings $A=\F[x_1,\ldots, x_n]/(x_1^{d_1}, \ldots, x_n^{d_n})$, where $d_1,\ldots, d_n\geq 1$ are integers, are called {\em monomial complete intersections}.
\end{ex}

\begin{exer}
Prove that the rings in \cref{ex: monomial CI} are the only artinian Gorenstein rings of the form $R/I$ where  $R=\F[x_1,\ldots, x_n]$ and $I$ is an ideal generated by monomials.
\end{exer}

All CI rings are Gorenstein, but not all Gorenstein rings are CI, as exemplified by the ring in \cref{exer: Gor}.

\section{The Lefschetz properties} \label{s: 3}

\subsection{Weak Lefschetz property and consequences}

\begin{defn}[Weak Lefschetz property] Let $A=\bigoplus_{i=0}^cA_i$ be a graded artinian $\F$-algebra. We say
that $A$ has the {\bf \em weak Lefschetz property (WLP)} if there exists an element $L\in A_1$
such that  for  $0\leq i \leq c-1$ each of the multiplication maps 
\[
\times L: A_i\to A_{i+1}, x\mapsto Lx \quad \text{ is injective or surjective}.
\] We call $L$ with this property a {\bf \em weak Lefschetz element}.
The WLP says that $\times L$ has the maximum possible rank, which is referred to as {\em full rank}. 
\end{defn}

\begin{defn}
The {\bf \em non-weak Lefschetz locus} of a graded artinian $\F$-algebra $A$ is the set (more accurately the algebraic set)
\[
NLL_w(A)=\{(a_1,\ldots, a_n)\in \F^n \mid L=a_1x_1+\cdots+a_nx_n \text{ not a weak Lefschetz element on } A\}.
\]
\end{defn}

\begin{exer}[Equivalent WLP statements]\label{exer: equivalent WLP}
Prove that for an artinian graded $\F$-algebra $A$ the following are equivalent:
\begin{enumerate}
\item $L\in A_1$ is a weak Lefschetz element for $A$.
\item For all $0\leq i \leq c-1$ the map $\times L: A_i\to A_{i+1}$  has rank   $\min\{h_A(i),h_A(i+1)\}$. 
\item For all $0\leq i \leq c-1$  $\dim_\F([(L)]_{i+1})=\min\{h_A(i),h_A(i+1)\}$. 
\item For all $0\leq i \leq c-1$  $\dim_\F([A/(L)]_{i+1})=\max\{0,h_A(i+1)-h_A(i)\}$. 
\item For all $0\leq i \leq c-1$  $\dim_\F([0:_A L]_{i})=\max\{0,h_A(i)-h_A(i+1)\}$. 
\end{enumerate}
\end{exer}

\begin{exer}
Show that the non-weak Lefschetz locus is a Zariski closed set.
\end{exer}

\begin{ex}
Take $A=\C[x,y]/(x^2,y^2)$ with the standard grading $|x|=|y|=1$ and $L=x+y$. Then the multiplication map $\times L$ gives the following matrices with respect to the monomial bases $\{1\},\{x,y\}$ and $\{xy\}$:

\begin{center}
\begin{tabular}{llll}
map & matrix & rank & injective/ surjective \\
\hline 
$A_0\to A_1$ &$\begin{bmatrix} 1 \\ 1\end{bmatrix}$ & 1 & injective\\
$A_1\to A_2$ & $\begin{bmatrix} 1 & 1\end{bmatrix}$ & 1 & surjective\\
$A_i\to A_{i+1}, i\geq 2 $ & $\begin{bmatrix} 0\end{bmatrix}$ & 0 & surjective\\
\end{tabular}
\end{center}

We conclude that $A$ has the WLP and $x+y$ is a Lefschetz element on $A$. 
\end{ex}

\begin{ex}[Dependence on characteristic]
Take $A=\F[x,y,z]/(x^2,y^2,z^2)$ with the standard grading $|x|=|y|=1$ and $L=ax+by+cz$. Then the multiplication map $\times L$ is represented by the following matrices with respect to the monomial bases $1$ for $A_0$, 
$\{x,y,z\}$ for $A_1$, $\{xy,xz,yz\}$ for $A_2$, and $xyz$ for $A_3$:
\begin{eqnarray*}
\times L:A_0\to A_1&
M=\begin{bmatrix}
a\\
b\\
c
\end{bmatrix}, & \text{ injective unless } a=b=c=0
\\
\times L:A_1\to A_2 &
M=\begin{bmatrix}
b & a & 0\\
c & 0 & a \\
0 & c & b
\end{bmatrix}, & \det(M)=-2abc\\
\times L:A_2\to A_3&
M=\begin{bmatrix}
a& b& c
\end{bmatrix}, & \text{ surjective unless } a=b=c=0.
\end{eqnarray*}

The map $\times L:A_1\to A_2$ has full rank if and only if $\chr(\F)\neq 2$ and $a\neq 0, b\neq 0, c\neq 0$. We conclude that $A$ has the WLP  if and only if  $\chr(\F)\neq 2$ because in that case  $L=x+y+z$ is a weak Lefschetz element. 

The {\em non-(weak) Lefschetz locus} of $A$ in this example is 
\begin{eqnarray*}
\NLL_w(A) &=& \{(a,b,c)\in \F^3 \mid L=ax+by+cz \text{ is not a weak Lefschetz element on } A\} \\
 &=& V(abc)= \{(a,b,c)\in \F^3 \mid a=0 \text{ or } b=0 \text{ or } c=0\} \\
 &=& \text{ the union of the three coordinate planes in } \F^3.
\end{eqnarray*}
\end{ex}

\begin{defn}
A sequence of numbers $h_1,\ldots,h_c$ is called {\em unimodal} if there is an integer $j$ such that
$$h_1\leq h_2 \leq \dots \leq h_j \geq h_{j+1}\geq \dots \geq h_c.$$
\end{defn}

\begin{lem}
If $B$ is a standard graded $\F$-algebra and $B_j=0$ for some $j\in\N$ then $B_i=0$ for all $i\geq j$.
\end{lem}
\begin{proof}
$B$ standard graded means that $B=\F[B_1]=\F[x_1,\ldots,x_n]/I$ where $x_1,\ldots,x_n$ are an $\F$-basis for $B_1$, which means  $|x_1|=\dots=|x_n|=1$, and $I$ is a homogeneous ideal.

Then it follows that $B_i=\Span_\F\{B_{i-j}B_j\}=\Span_\F\{0\}=0$ for any $i\geq j$.
\end{proof}

\begin{prop}Suppose that $A$ is a standard graded artinian algebra over a
field $\F$. If $A$ has the weak Lefschetz property then $A$ has a unimodal Hilbert
function.
\end{prop}
\begin{proof}
Let $j$ be the smallest integer such that $\dim_\F A_j > \dim_\F A_{j+1}$ and let $L$ be a Lefschetz element on $A$. Then $\times L:A_j\to A_{j+1}$ is surjective i.e. $LA_j=A_{j+1}$. Now consider the cokernel $A/(L)$ of the map
$A\stackrel{\times L}{\lra}A$. 
We have that $\left(A/(L)\right)_{j+1}=A_{j+1}/LA_j=0$, so by the previous Lemma $\left[A/(L)\right]_{i}=0$ for $i\geq j+1$. This means that $\times L: A_i\to A_{i+1}$ is surjective for $i\geq j$ and  so we have
$$h_0(A)\leq h_1(A)\leq \dots \leq h_j(A)>h_{j+1}(A)\geq h_{j+2}(A)\geq \dots \geq h_c(A).  \qedhere$$
\end{proof}

The proof above yields: 
\begin{cor}
 For a standard graded artinian algebra $A$ with the weak Lefschetz property there exists $j\in \N$ such that the multiplication maps by a weak Lefschetz element $\times L:A_i\to A_{i+1}$ are injective for $i<j$  after which they become surjective for $i\geq j$.
\end{cor}

\begin{ex}[Dependence on grading]
Recall from Example 2.18 that the algebra $A=\F[x,y]/(x^2,y^2)$ with $|x|=|y|=1$ is standard graded and has WLP and notice that the Hilbert function of $A$, $1,2,1$ is unimodal.

Take $B=\C[x,y]/(x^2,y^2)$ with  $|x|=1, |y|=3$. Then $B$ is a graded algebra with nonunimodal Hilbert function $1, 1, 0, 1, 1$, but  $x$ is a weak Lefschetz element on $B$. 

Take $C=\C[x,y]/(x^2,y^2)$ with $|x|=1,|y|=2$. Then $C$ has a unimodal Hilbert function
$1,1,1,1$ but does not have the WLP because the only degree one elements are multiples of $x$, which do not have maximal rank with respect to multiplication from degree one to degree two. 
\end{ex}

\subsection{Strong Lefschetz property and consequences}

\begin{defn}[Strong Lefschetz property] Let $A=\bigoplus_{i=1}^cA_i$ be a graded artinian $\F$-algebra. We say
that $A$ has the {\bf \em strong Lefschetz property (SLP)} if there exists an element $L\in A_1$
such that for all $1\leq d\leq c$ and $0\leq i \leq c-d$ each of the multiplication maps 
\[
\times L^d: A_i\to A_{i+d}, x\mapsto L^dx \quad \text{ is  injective or surjective.}
\] 
 We call $L$ with this property a {\bf \em strong Lefschetz element}.
\end{defn}

\begin{rem}
An element $L\in A_1$ is strong Lefschetz on $A$ if and only if  for all $1\leq d\leq c$ and $0\leq i \leq c-d$ the maps $\times L^d: A_i\to A_{i+d}$ have rank  $\min\{h_A(i),h_A(d+i)\}$.
\end{rem}

\begin{defn}
The {\bf \em non-strong Lefschetz locus} of a graded artinian $\F$-algebra $A$ is the set (more accurately the algebraic set)
\[
NLL_s(A)=\{(a_1,\ldots, a_n)\in \F^n \mid L=a_1x_1+\cdots+a_nx_n \text{ not a strong Lefschetz element on } A\}.
\]
\end{defn}

\begin{rem}
The non-strong Lefschetz locus is a Zariski closed set.
\end{rem}

\begin{rem}[SLP $\Rightarrow$ WLP]
If $A$ satisfies SLP then $A$ satisfies WLP (the $d=1$ case).
\end{rem}

The following exercise shows this implication is not reversible.

\begin{exer}
Let $\F$ be a field of characteristic zero and let
\[
A=\frac{\F[x,y,z]}{(x^3,y^3,z^3,(x+y+z)^3)}.
\]
\begin{enumerate}
\item Find the Hilbert function of $A$.
\item Prove that $A$ satisfies WLP but not SLP.
\end{enumerate}
{\em Hint:} One can use a computer algebra system such as {\em Macaulay2} \cite{M2} to prove  that an algebra satisfies WLP or SLP by finding a linear form that conforms to the respective definition. Usually such form is found by trial and picking at random from the set of linear forms. Macaulay2 has a function \texttt{random()} which is helpful in this regard. 

One can also prove  that an algebra does or does not satisfy WLP or SLP by working over an extension of the coefficient field $\F$ of $A$. In the example above, one would work over the filed extension $K=\F(a,b,c)$, defined in Macaulay2 as \texttt{K=frac(QQ[a,b,c])}. To compute the rank of multiplication by $L=ax+by+cz$ and its powers on the artinian algebra $A'=K\otimes_{\F}A$ in Macaulay2 for $\F=\Q$ use the criteria in \Cref{exer: equivalent WLP} to express the desired ranks in terms of the Hilbert function of cyclic modules $A'/(L^j)$.
\end{exer}

\begin{ex}[Dependence on characteristic]
Take $A=\F[x,y]/(x^2,y^2)$ with the standard grading $|x|=|y|=1$ and $L=ax+by$. Then the multiplication map $\times L^2$ gives the following matrices with respect to the monomial bases $\{1\},\{x,y\}$ and $\{xy\}$:

\begin{center}
\begin{tabular}{llll}
map & matrix & rank & injective/ surjective \\
\hline 
$A_0\to A_2$ &$\begin{bmatrix} 2ab \end{bmatrix}$ & $\begin{cases} 1 & \chr(\F)\neq 2 \\ 0 & \chr(\F) = 2 \end{cases}$ & $\begin{cases} \text{bijective} & \chr(\F)\neq 2 \\ \text{none} & \chr(\F) = 2 \end{cases}$\\
$A_i\to A_{i+2}, i\geq 1 $ & $\begin{bmatrix} 0\end{bmatrix}$ & 0 & surj\\
\end{tabular}
\end{center}

\noindent If $\chr(\F)\neq 2$ we conclude that $A$ has the SLP and $ax+by$ where $a\neq 0, b\neq 0$ is a Lefschetz element on $A$. The {\em non-(strong) Lefschetz locus} is the union of the coordinate axes in $\F^2$
$$NLL_s(A)=V(ab)=\{(a,b)\in \F^2 \mid a=0 \text{ or } b=0\}.$$

However $A$ does not have the SLP if $\chr(\F)=2$ so in that case $\NLL_s(A)=\F^2$.
\end{ex}

\begin{prop} Let $A$ be a (not necessarily standard) graded artinian $\F$-algebra which satisfies the SLP. Then $A$ has unimodal Hilbert function.
\end{prop}
\begin{proof} Suppose that the Hilbert function of $A$ is not unimodal. Then there are
integers $k < l <m$ such that $\dim_\F A_k > \dim_F A_l < \dim_\F A_m$. Recall the rank of a composite map is bounded above by the ranks of its components. 
Hence the multiplication map $\times L^{m-k}:A_k \to A_m$ cannot have full rank for any
linear element $L\in A$ because it is the composition of $\times L^{m-l}:A_l \to A_m$ and $\times L^{l-k}:A_l \to A_k$, each of which have rank strictly less than $\min\{\dim_\F A_k,  \dim_\F A_m\}$. Thus $A$ cannot have the SLP.
\end{proof}

\begin{defn}. Let $A=\bigoplus_{i=1}^cA_i$ be a graded artinian $\F$-algebra. We say
that $A$ has the {\bf \em strong Lefschetz property in the narrow sense  (SLPn)} if there exists an element $L\in A_1$ such that the multiplication maps $\times L^{c-2i}: A_i\to A_{c-i}, x\mapsto L^{c-2i}x$ are bijections for all $0\leq i \leq \lceil c/2\rceil$. 
\end{defn}

\begin{rem}
SLP in the narrow sense is the closest property to the conclusion of the Hard Lefschez  \cref{HLT}.
\end{rem}

\begin{defn}
We say that a graded artinian algebra $A=\bigoplus_{i=1}^c A_i$ of maximum socle degree $c$ has a {\em symmetric Hilbert function} if $h_A(i)=h_A(c-i)$ for $0\leq i \leq \lceil c/2\rceil$.
\end{defn}

\begin{prop}\label{equiv narrow SLP}
 If a graded artinian $\F$-algebra $A$ has the strong Lefschetz property in the narrow
sense, then the Hilbert function of A is unimodal and symmetric. Moreover we have
the equivalence: $A$ has SLP and  symmetric Hilbert function if and only if $A$ has SLP in the narrow sense.
\end{prop}

\begin{proof}
($\Leftarrow$)
The fact that  SLP in the narrow sense implies symmetric Hilbert function follows from the definition because the bijections give $\dim_F A_i=\dim_F A_{c-i}$.

The fact that  SLP in the narrow sense implies SLP can be noticed by considering $\times L^d:A_i\to A_{i+d}$. For each such $d,i$ there exists $j=c-i-d$ such that:
\begin{itemize}
\item if $i\leq (c-d)/2$ then $j=c-i-d \geq i$ and $(\times L^{d}) \circ (\times L^{j-i}) =\times L^{c-2i}$ is a bijection implies that $\times L^{d}$ is surjective, hence has full rank;
\item if $i> (c-d)/2$ then $c-2i< d$ and $(\times L^{d-(c-2i)})\circ (\times L^{d})=\times L^{c-2i}$ is a bijection implies that $\times L^{d}$ is injective, hence full rank;
\end{itemize}

($\Rightarrow$) 
The fact that SLP + symmetric Hilbert function implies SLPn is clear from the definitions. 
\end{proof}

\begin{ex}
The standard graded algebra $\F[x,y]/(x^2,xy,y^a)$ with $a>3$ has non-symmetric Hilbert function $1,2,1^{a-2}$ (here $1^{a-2}$ means $1$ repeated $a-2$ times). Notice that $A$ has the SLP because $L=x+y$ is a strong Lefschetz element. However, $A$ does not satisfy SLPn because its Hilbert function is not symmetric.
\end{ex}

\subsection{Stanley's Theorem}

The most famous theorem in the area of investigation of the algebraic Lefschetz properties, and also the theorem which started this, is the following:

\begin{thm}[Stanley's theorem]
\label{thm:Stanley}
If $\chr(\F)=0$, then all monomial complete intersections, i.e. $\F$-algebras of the form 
$$A=\frac{\F[x_1,\ldots, x_n]}{(x_1^{d_1},\ldots,x_n^{d_n})}$$
 with $d_1,\ldots,d_n\in \N$ have the SLP.
\end{thm}
\begin{proof}
Recall that $H^\bullet(\P_\C^{d-1},\F)=\F[x]/(x^d)$, so by K\"unneth we have
$$H^\bullet(\P_\C^{d_1-1}\times \P_\C^{d_2-1}\times \cdots \times \P_\C^{d_n-1},\F)=\F[x]_1/(x_1^{d_1})\otimes_\F\F[x_2]/(x_2^{d_2})\otimes_\F \cdots\otimes_\F \F[x_n]/(x_n^{d_n})=A.$$
Since $X=\P_\C^{d_1-1}\times \P_\C^{d_2-1}\times \cdots \times \P_\C^{d_n-1}$ is an irreducible complex projective variety, the Hard Lefschetz theorem says that $A$ has SLP in the narrow sense which implies that $A$ has SLP.
\end{proof}

Alternate proofs of this theorem have been given by Watanabe in \cite{Watanabe} using representation theory, by Reid, Roberts and Roitman in \cite{RRR}  with algebraic methods, also by Lindsey \cite{Lindsey} and by Herzog and Popescu \cite{HP}. We will give a different proof of Stanley's theorem later in these notes in \Cref{Stanley second}.

\begin{exer}
With help from a computer make conjectures regarding the WLP and SLP for monomial complete intersections in positive characteristics. A characterization is known for SLP, but not for WLP. See \cite{Cook, Nicklasson} for related work.
\end{exer}

\subsection{Combinatorial applications}
The following discussion of a spectacular application of SLP is taken from \cite{StanleyICM}.

A {\em polytope} is a convex body in Euclidean space which is bounded and has finitely many vertices. 
Let $\cP$  be a $d$-dimensional simplicial convex polytope with $f_i$ $i$-dimensional faces, $0<i<d-1$. We call the vector $f(\cP)= (f_0,\ldots,f_{d-1})$ the {\em $f$-vector} of $\cP$.The problem of obtaining information about such vectors goes back to Descartes and Euler. In 1971 McMullen \cite{McMullen} gave a remarkable condition on a vector $(f_0,\ldots,f_{d-1})$ which he conjectured was equivalent to being the $f$-vector of some polytope.

To describe this condition, define a new vector $h(\cP) = (h_0, ...,h_d)$,
called the {\em $h$-vector} of $\cP$, by
\[
h_i=\sum_{j=0}^i \binom{d-j}{d-i}(-1)^{i-j}f_{j-1}
\]
where we set $f_{-1} = 1$.  McMullen's conditions, though he did not realize it, turn out to be equivalent to $h_i=h_{d-i}$ for all $i$ together with the existence of a standard graded commutative algebra $A$ with $A_0=\F$ and  $h_A(i)=h_i-h_{i-1}$ for $1\leq i<\lfloor d/2\rfloor$.

Stanley \cite{Stanley} constructed from $\cP$ a $d$-dimensional complex projective toric variety $X(\cP)$ for which $\dim_\C H^{2i}(X(\cP))= h_i$. Although $X(\cP)$ need not be smooth, its singularities are sufficiently nice that the hard Lefschetz theorem continues to hold. Namely, $X(\cP)$ locally looks like $\C^n/G$, where $G$ is a finite group of linear transformations. Ta\-king $A=H^{*}(X(\cP))/(L)$ with degrees scaled by $1/2$, where $L$ is the class of a hyperplane section, the necessity of McMullen's condition follows from \cref{exer: equivalent WLP}(4). Sufficiency was proved about the same time by Billera  and Lee \cite{BL}. 



\section{Lefschetz property via representation theory of $\sl_2$} \label{s: 4}

\subsection{The Lie algebra $\sl_2$ and its representations} 

\phantom{a}

Some of the exercises in this section are taken from \cite{Robles}.

Throughout this section let $\F$ be an algebraically closed field of characteristic zero.

\begin{defn}
A {\em Lie algebra} is a vector space $\fg$ equipped with a bilinear operator $[-,-]:\fg\times \fg \to \fg$ satisfying the following two conditions : 
\begin{itemize}
\item $[x,y]=-[y,x]$
\item $[[x,y],z]+[[y,z],x]+[[z,x],y]=0$.
\end{itemize}
The bilinear operator $[-,-]$ is called the bracket product, or simply the {\em bracket}. The
second identity in the definition is called the {\em Jacobi identity}.
\end{defn}

\begin{rem}
Any associative algebra has a Lie algebra structure with the bracket product
defined by commutator $[x,y]=xy-yx$. Associativity implies the Jacobi identity. 
\end{rem}

 The set of $n\times n$ matrices $\mathcal{M}_n(\F)$ forms a Lie algebra since it is associative. This Lie algebra is denoted by $\gl_n(\F)$. 

\begin{defn}
Since the set of matrices of trace zero is closed under the bracket of $\gl_n(\F)$ (because $\tr(AB)=\tr(BA)$ for any matrices $A,B$), it forms a Lie subalgebra
\[
\sl_n(\F)=\{M\in \gl_n(\F) \mid \tr(M)=0\}.
\]
\end{defn}

\begin{ex}[The Lie algebra $\sl_2(\F)$]
In the case where $n=2$,  $\sl_2(\F)$ is three-dimensional, with basis
\[
E=
\begin{bmatrix} 
0 &1\\
0 &0
\end{bmatrix},
\quad
H=
\begin{bmatrix}
1 & 0\\
0 &-1
\end{bmatrix},
\quad
F=\begin{bmatrix}
0 &0 \\
1 & 0
\end{bmatrix}
\]
The three elements
$E,H,F$ are called the {\em $\sl_2$-triple}.
These elements satisfy the following three relations, which we
call the fundamental relations:
\begin{equation}
\label{eq:sl2}
[E,F]=H, \quad [H,E]=2E, \quad [H,F]=-2F.
\end{equation}
The algebra $\sl_2(\F)$ is completely determined by these relations. \end{ex}

We are interested in representations of $\sl_2$.

\begin{defn}[Lie algebra representation]
Let $V$ be an $\F$-vector space.  Then  $\End(V)$ is a Lie algebra with the bracket defined by $[f,g]=f\circ g-g\circ f$. A {\em representation} of a Lie algebra $\fg$ is vector space $V$ endowed with a Lie algebra homomorphism 
\[\rho:\fg\to \End(V),\]
i.e. a vector space homomorphism which satisfies 
\[\rho([x,y])=[\rho(x),\rho(y)].\]

A representation is called {\em irreducible} if it contains no trivial (nonzero) subrepresentation i.e. if $W\subsetneq V$ is such that $\rho(W)\subseteq W$ then $W=0$.
\end{defn}

In the case of $\fg=\sl_2(\F)$, we abuse notation and call the set of elements $\rho(E), \rho(H), \rho(F)$ just $E, H, F$ and say they form an {\em $\sl_2$-triple}.

\begin{exer}\label{ex: Sd irrep}
Let $\F[x,y]_d$ be the vector space of homogeneous polynomials of degree $d$ in $\F[x,y]$. Prove that
\begin{enumerate}
\item $E=x\frac{\partial}{\partial y}, H=x\frac{\partial}{\partial x}-y\frac{\partial}{\partial y}, F=y\frac{\partial}{\partial x}$ form an $\sl_2$-triple.
\item Prove that the monomial $x^ay^b$ is an eigenvector of $H$ with eigenvalue $a-b\in\Z$. In particular the eigenvalues of $H$ on $\F[x,y]_d$ are $d, d-2, d-4, \ldots, 4-d, 2-d, -d$.
\item Prove that a basis of $\F[x,y]_d$ is $y^d, E(y^d), E^2(y^d), \ldots, E^d(y^d)$.
\end{enumerate}
Pictorially this can be summarized as 

\smallskip
\begin{tikzpicture}[->,>=stealth',auto,node distance=2.5cm, thick]

  \node (1) {$0$};
  \node (2) [right of=1] {$\F y^d$};
  \node (3) [right of=2] {$\F xy^{d-1}$};
  \node (4) [right of=3] {$\cdots$};
    \node (5) [right of=4] {$\F x^{d-1}y$};
      \node (6) [right of=5] {$\F x^d$};
        \node (7) [right of=6] {$0$};

  \path[every node/.style={font=\small}]
    (2) edge[bend right] node [midway,below] {E} (3)
    (3) edge[bend right] node [midway,below] {E} (4)
    (4) edge[bend right] node [midway,below] {E} (5)
    (5) edge[bend right] node [midway,below] {E} (6)
    (6) edge node [midway,below] {E} (7)
    (2) edge node [midway,above] {F} (1)
    (3) edge[bend right] node [midway,above] {F} (2)
    (4) edge[bend right] node [midway,above] {F} (3)
    (5) edge[bend right] node [midway,above] {F} (4)
    (6) edge[bend right] node [midway,above] {F} (5);

\end{tikzpicture}
\end{exer}

We will soon see that the vector space in \cref{ex: Sd irrep} is the basic building block of all other representations of $\sl_2$.

 An important result on Lie algebra representations  is:
\begin{thm}[Weyl's Theorem]
Any  Lie algebra representation decomposes uniquely up to isomorphism as a direct sum of irreducible representations.
\end{thm}

\begin{defn}[Weight vectors]
Let $\rho:\sl_2(\F)\to \End(V)$ be a representation. The eigenvalues of $H$ are called {\em weights} and the eigenvectors are called {\em weight vectors}. In particular an eigenvector $u$ is called a {\em lowest weight vector} if $Fu=0$ and is called a {\em highest weight vector} if $Eu=0$. 
\end{defn}

\begin{ex}
In the representation introduced in \cref{ex: Sd irrep} the highest weight vectors are the elements of $\F x^d$ and  the lowest weight vectors are the elements of $\F y^d$.
\end{ex}

To justify the name of highest weight we state the following theorem:

\begin{thm}[Irreducible representations of $\sl_2$]
\label{thm:sl2reps}
Let $\rho:\sl_2(\F)\to \End(V)$ be an irreducible representation with $\dim(V)=d+1$. Then there exist a basis $\cB=\{v_0, \ldots, v_d\}$ for $V$  such that 
\begin{enumerate}
\item each $v_i$ is an eigenvector for $H$ with eigenvalue $-d+2i$, i.e. $Hv_i=(-d+2i)v_i$
\item $Ev_i=v_{i+1}$ for $i<d$, $Ev_d=0$
\item $Fv_i=i(d-i+1)v_{i-1}$ for $i>0$, $Fv_0=0$.
\end{enumerate}
In particular, the elements $E,H,F\in M_{d+1}(\F)$ are represented with respect to this basis by the matrices
\begin{eqnarray}
[E]_\cB &=&
\begin{bmatrix}
0& 0 & \cdots & 0 &0 \\
1& 0& \cdots & 0 &0 \\
\vdots & \ddots &\ddots & \vdots & \vdots \\
0 & 0 &\cdots & 1 &0
\end{bmatrix}, 
\\
{[H]_{\cB}} &=&
\begin{bmatrix}
-d& 0 & \cdots & 0  \\
0& -d+2& \cdots & 0 \\
\vdots & \ddots &  & \vdots \\
0 & 0 &\cdots & d
\end{bmatrix},  \label{eq: H}
\\
{[F]_\cB} &=&
\begin{bmatrix}
0& 1\cdot d & \cdots & 0 &0 \\
0& 0& 2(d-1) & 0 &0 \\
\vdots & \ddots &\ddots & \vdots & \vdots \\
0 & 0& \cdots&0 & d\cdot 1 \\
0 & 0 &\cdots & 0 &0
\end{bmatrix}.\label{eq: F}
\end{eqnarray}
\end{thm}

\begin{exer}
Find a basis that satisfies the properties given  by \cref{thm:sl2reps} for the representation $\F[x,y]_d$ introduced in \cref{ex: Sd irrep}.
\end{exer}

\cref{thm:sl2reps} above says in particular that there is only one representation of $\sl_2(\F)$ of dimension $d+1$ (up to isomorphism). A representative for this isomorphism class can be chosen to be the representation $\F[x,y]_d$ in \cref{ex: Sd irrep}. Furthermore any representation of $\sl_2(\F)$ has a basis consisting of weight vectors. This justifies the following:

\begin{defn}
Let $V$ be a representation of $\sl_2$ and let $W_\lambda(V)=\{v\in V \mid  Hv=\lambda v\}$ be the eigenspace corresponding to a weight (eigenvalue) $\lambda$ for $H$. Then there is a decomposition $V=\bigoplus_{\lambda} W_\lambda(V)$  called the {\em weight space decomposition} of $V$.
\end{defn}

\begin{rem}
If $V$ is an irreducible representation for $\sl_2(\F)$ and $\dim(V)=n+1$ then the weight spaces are the 1-dimensional spaces $W_{-n+2i}(V)=\F v_i$, with $v_i$ as in \cref{thm:sl2reps}.
\end{rem}

\begin{exer}
\begin{enumerate} 
\item Suppose that $V$ is a representation of  $\sl_2$ and that the eigenvalues of $H$ on $V$ are $2, 1, 1, 0, -1, -1, -2$. Show that the irreducible decomposition of $V$ is $V\cong \F[x,y]_2\oplus \F[x,y]_1\oplus \F[x,y]_1$.
\item Prove that if $V$ is any representation of  $\sl_2$ then its irreducible decomposition is determined by the eigenvalues of $H$.
\end{enumerate}
\end{exer}

\begin{exer}
Let $V$ be an $\sl_2$ representation and set $W_k=\{v\in V\mid H(v)=kv\}$.
\begin{enumerate} 
\item Show that $\dim_\F W_k=\dim_\F W_{-k}$.
\item Prove that $E^k:W_{-k}\to W_k$ is an isomorphism.
\item  Show that $\dim_\F W_{k+2}\leq \dim_\F W_k$ for all $k\geq 0$, that is, the two sequences
\[
\ldots, \dim_\F W_4,   \dim_\F W_2,  \dim_\F W_0,  \dim_\F W_{-2},  \dim_\F W_{-4},\ldots 
\]
\[
\ldots, \dim_\F W_3,   \dim_\F W_1,  \dim_\F W_{-1},  \dim_\F W_{-3},  \ldots 
\]
are unimodal.
\end{enumerate}
\end{exer}

\subsection{Weight space decompositions and the narrow sense of SLP}

We now show that there is a close connection between artinian algebras satisfying SLPn and the representations of $\sl_2$.

\begin{rem}
If $A$ is a graded artinian $\F$-algebra and $L$ is a linear form, then we can view $A$ as a $\F[L]$-module since by the universal mapping property of polynomial rings there exists a well defined ring homomorphism $\F[L]\to A$ which maps $L\mapsto L$. Since $\F[L]$ is a PID and $A$ is a module over it, the structure theorem for modules over PIDs says that there is a module isomorphims
\[
A\cong  \F[L] / ( p_ 1^{e_1} ) \oplus \cdots \oplus F[L] / ( p_k^{e_k} ) 
\] 
where each $p_i$ is a prime element of $\F[L]$ (no free part since $A$ is finite dimensional). Since $A$ is furthermore graded the elementary divisors $p_i^{e_i}$ must be homogeneous elements of $\F[L]$, thus $p_i=L$ for all $i$.
This implies that $A$ decomposes as a direct sum
\[
A\cong  S^{(1)} \oplus \cdots \oplus S^{(k)}, \text { with } S^{(i)}\cong  \F[L] / ( L^{e_i} ).
\] 
The cyclic $\F[L]$ modules $ S^{(i)}$ are the {\em strands} of multiplication by $L$ on $A$ which are essential in defining the notion of Jordan type, a finer invariant than the Lefschetz properties. The generic {\em Jordan type} of $A$ is the multi-set of sizes of Jordan blocks $e_1, e_2, \ldots, e_s$ ordered decreasingly. This follows because the action of $L$ on  $ S^{(i)}$ is given by a single Jordan block of size $e_i$.
\end{rem}

Here is the connection between SLPn and the representations of $\sl_2(\F)$: 
\begin{cor}
The following are equivalent
\begin{enumerate}
\item $S$ is a cyclic graded $\F[L]$ module i.e. $S\cong  \F[L] / ( L^{d} )$ (not necessarily degree preserving isomorphism)
\item $S\cong \F[x,y]_{d-1}$ as an irreducible representation of $\sl_2$ with $Es=Ls$. 
\end{enumerate}
\end{cor}
\begin{proof}
This follows because both the action of $L$ on $S$ as well as the action of $E$ on $\F[x,y]_{d-1}$ is given by a single Jordan block matrix. Once the basis of $S$ has been fixed to be $1,L, L^2, \ldots, L^{d-1}$, the action of $H$ and $F$ can be simply defined to be the one given by the matrices displayed in \cref{thm:sl2reps}.
\end{proof}

If we put the $\sl_2(\F)$-module structures on the individual strands together we obtain:

\begin{thm}[SLPn via weight decomposition]\label{thm:sl2andSLPn}
Let $A$ be a graded artinian algebra of socle degree $c$ and let $L\in A_1$. The following are equivalent
\begin{enumerate}
\item $L$ is a strong Lefschetz element on $A$ in the narrow sense,
\item $A$ is an $\sl_2(\F)$-representation with $E=\times L$ and the weight space decomposition of $A$ coincides with the grading decomposition via ${\rm weight}(v)=2\deg(v)-c$.
This means that 
\[
A=\bigoplus_{i=0}^c A_i=\bigoplus_{i=0}^c W_{2i-c}(A), \text{ where } A_i=W_{2i-c}(A).
\]
\end{enumerate}
\end{thm}

\begin{proof}[Proof sketch]

Suppose $L$ is a strong Lefschetz element on $A$ in the narrow sense. We construct an $\sl_2(\F)$ triple in $\End_\F(A)$ as follows: let $E=\times L:A\to A$. Consider the Jordan decomposition of $A$ with respect to the endomorphism $E$ written as $A=\bigoplus V_i$, that is, let the subspaces $V_i$ be the eigenspaces of the operator $E$. For each $V_i$, let  $F_i,H_i:V_i\to V_i$ to be the endomorphisms of $V_i$ given with respect to the basis in which $E\big |_{V_i}$ is in Jordan form by the matrices in \eqref{eq: H} and \eqref{eq: F}, respectively, where $d=\dim(V_i)-1$. Setting $H=\bigoplus H_i$ and $F=\bigoplus F_i$, one can check $E, H, F$ is an $\sl_2(\F)$ triple. Furthermore, from the properties of Jordan type one knows that the Jordan blocks of $L$ are centered around the middle degree of $A$; see \cite[Proposition 2.38]{IMM}. It follows that if $v$ is an eigenvector of weight $2k-d$ it is in degree $ (c-d)/2 +k$ (note that $c\equiv d\pmod{2}$). Substituting $i=(c-d)/2 +k$   it follows that $W_{2i-c}(A)=A_i$.

Conversely, suppose $A$ is an $\sl_2(\F)$-representation with $E=\times L$. Then one can use the information about the grading to verify that the Jordan blocks are centered around degree $\lfloor c/2\rfloor$. Thus the Jordan degree type is the transpose of the Hilbert function of $A$. By \cite[Proposition 3.64 (2)]{book} it follows that $L$ is a strong Lefschetz element on $A$ in the narrow sense.
\end{proof}

\subsection{Tensor products}\label{s: 5}
From \cref{thm:sl2andSLPn}, we can deduce how SLPn behaves when we take tensor products. We need the following lemma.

\begin{lem}
If $\F$ is an algebraically closed field of characteristic zero and $A,A'$ are associative algebras which are representations of $\sl_2(\F)$ , then so is $A\otimes_\F A'$ with the action $g\cdot(v\otimes v')=(gv)\otimes v' + v\otimes (gv')$. If $v,v'$ are weight vectors then $v\otimes v'$ is also a weight vector with $\weight(v\otimes v')=\weight(v)+\weight(v')$.
\end{lem}
\begin{proof}
We show the statement about weights only: say $\weight(v)=\lambda$ and  $\weight(v')=\lambda'$ so that
$Hv=\lambda v, Hv'=\lambda v'$. Then 
\[
H(v\otimes v')=(Hv)\otimes v' + v\otimes (Hv')=\lambda v\otimes v'+v\otimes \lambda'v'=(\lambda+\lambda')v\otimes v'
\]
shows that $v\otimes v'$ is a weight vector with weight $\lambda+\lambda'$.
\end{proof}

\begin{thm}
\label{thm:tensor} 
Let $\F$ be an algebraically closed  field of characteristic zero.
If $L$ is a strong Lefschetz element in the narrow sense on $A$ and if $L'$ is a strong Lefschetz element in the narrow sense on $A'$ then $L\otimes 1+1\otimes L'$ is a strong Lefschetz element in the narrow sense on $A\otimes_\F A'$.
\end{thm}
\begin{proof}
By  \cref{thm:sl2andSLPn} we have that if $c, c'$ are the socle degrees of $A,A'$, respectively, then $A_i=W_{2i-c}(A)$ and $A'_j=W_{2j-c'}(A')$, so
\begin{equation}\label{eq: 1}
A=\bigoplus_{i=0}^c A_i=\bigoplus_{i=0}^c W_{2i-c}(A) \quad \text{and} \quad A'=\bigoplus_{j=0}^{c'} A'_j=\bigoplus_{j=0}^{c'} W_{2j-c'}(A')
\end{equation}
imply 
\begin{equation}\label{eq: 2}
A\otimes_\F A'=\bigoplus_{i=0, j=0}^{c,c'} A_i\otimes_\F A'_j=\bigoplus_{i=0, j=0}^{c,c'} W_{2i-c}(A)\otimes_\F W_{2j-c'}(A').
\end{equation}
From the fact that $\deg(v\otimes v')=\deg(v)+\deg(v')$ and \eqref{eq: 1} we deduce that 
\[
(A\otimes_\F A')_k=\bigoplus_{i=0}^c A_i\otimes_F A'_{k-i}.
\]
Note that the maximum socle degree of $A\otimes_\F A'$ is $c+c'$. From the identity 
$$\weight(v\otimes v')=\weight(v)+\weight(v')$$ and \eqref{eq: 2} we deduce that
\[
W_{2k-c-c'}(A\otimes_\F A')=\bigoplus_{i=0}^c W_{2i-c}(A)\otimes_\F W_{2(k-i)-c'}(A')=\bigoplus_{i=0}^c A_i\otimes_\F A'_{k-i}.
\]
Comparing, we see that $(A\otimes_\F A')_k=W_{2k-c-c'}(A\otimes_\F A')$, where the weight spaces of $A\otimes_\F A'$ correspond to the action 
\[
E(v\otimes v')=Ev\otimes v'+v\otimes Ev'=Lv\otimes v'+v\otimes L'v'=(L\otimes 1+1\otimes L')v\otimes v'.
\]
\cref{thm:sl2andSLPn} gives that $L\otimes 1+1\otimes L'$ is a strong Lefschetz element on $A\otimes_\F A'$.
\end{proof}

A corollary of \cref{thm:tensor} is the following

\begin{cor}[Tensor product preserves SLPn]
If $\F$ is a  field of characteristic zero and $A,A'$ are graded artinian $\F$-algebras which satisfy SLPn, then $A\otimes_\F A'$ also satisfies SLPn.
\end{cor}

From the above corollary one can easily deduce Stanley's theorem applying induction on the embedding dimension $n$.

\begin{cor}[Stanley's Theorem - second proof]\label{Stanley second}
If $\F$ has characteristic 0, then the algebra $A=\F[x_1,\ldots,x_n]/(x_1^{d_1}, \ldots, x_n^{d_n})=\F[x_1]/(x_1^{d_1}) \otimes_\F \cdots \otimes_\F \F[x_n]/(x_n^{d_n}) $ satisfies SLP in the narrow sense.
\end{cor}

\begin{rem} \hfill
\begin{enumerate}
\item
While the symmetric unimodality of Hilbert functions  is preserved under taking tensor product, just unimodality is not.  For example for 
\[
A=\F[x,y,z]/(x^2,xy,y^2,xz,yz,z^5)
\] with Hilbert function $1, 3, 1, 1, 1$ we have that the Hilbert function of $A\otimes_\F A$ is $1, 6, 11, 8, 9, 8, 3, 2, 1$.
\item
While the SLPn is preserved under taking tensor product, the SLP (not in the narrow sense) is not preserved by tensor product. In the example above $A$ has SLP but since its Hilbert function is not unimodal, $A\otimes_\F A$ cannot have the SLP.
\end{enumerate}
\end{rem}

The issue in part 2 of the remark is remedied by restricting to Gorenstein algebras, which have symmetric Hilbert function. Recall that for algebras with symmetric Hilbert function the SLP is equivalent to SLPn by \Cref{equiv narrow SLP}. Thus we have:

 \begin{cor}[{\cite[Theorem 6.1]{MW}}]
If $\F$ is a  field of characteristic zero and $A,A'$ are graded artinian Gorenstein $\F$-algebras which satisfy SLP, then $A\otimes_\F A'$ also satisfies SLP.
\end{cor}

\section{Gorenstein rings via Macaulay inverse systems} \label{s: 6}

The description of the dual ring of the polynomial ring in \cref{DUAL} is taken from \cite{DNS}. The material in \cref{MIS} follows Eisenbud's Commutative Algebra book \cite{Eisenbud}  and Geramita's lectures \cite[Lecture~9]{Ger95}. The material on Hessians in \cref{HESS} follows \cite{MW}.

\subsection{The graded dual of the polynomial ring}\label{DUAL}

Recall the notion of a dual for an $\F$-vector space:

\begin{defn}
Let $V$ be an $\F$-vector space. Its dual is 
\[V^*=\Hom_\F(V,\F)=\{\varphi: V\to \F \mid \varphi \text{ is } \F-\text{linear}\},
\] the vector space of linear functionals on $V$.  
\end{defn}

\begin{exer}\label{exer: double dual}
If $V$ is a finite dimensional vector space, there  is a natural isomorphism of vector spaces $V\cong V^{**}$.
\end{exer}

We extend this idea to construct duals of rings and modules.

\begin{defn}[Divided power algebra]
\label{def:gradedhom}
Say $R=\F[x_1,\ldots,x_n]$ is the polynomial ring.  Let
$$R^*:=\Hom_\F^{\rm gr}(R,\F)=\bigoplus_{i\geq 0} \Hom_\F(R_i,\F).$$

 We use a standard shorthand for monomials: if $\mathbf{a}=(a_1,\ldots,a_n)\in \Z^{n}_{\ge 0}$, then $x^\mathbf{a}=x_1^{a_1}\cdots x_n^{a_n}$ is the corresponding monomial in $R$. If $x^\mathbf{a}$ is in $R_d$, we write $X^{[\mathbf{a}]}$ for the functional (in $R^*_d$) on $R_d$ which sends $x^{\mathbf{a}}$ to $1$ and all other monomials in $R_d$ to $0$. We'll make the convention from now on to write $X_i$ for the duals of the elements $x_i$ in $R_1^*$.  As a vector space, $R^*$ is isomorphic to a polynomial ring in the $n$ variables $X_1, \ldots, X_n$.  However, as we recall shortly, $R^*$ has the multiplicative structure of a \textit{divided power algebra}.  For this reason, we call $X^{[\mathbf{a}]}$ a \textit{divided} monomial and we write $R^*=\F[X_1,\ldots, X_n]_{DP}$. 
  \end{defn}

 The ring $R$ acts on $R^*$ by \textit{contraction}, which we denote by $\contract$.  That is, if $x^{\mathbf{a}}$ is a monomial in $R$ and $X^{[\mathbf{b}]}$ is a divided monomial in $R^*$, then
\[
x^{\mathbf{a}}\contract X^{[\mathbf{b}]}=\begin{cases} X^{[\mathbf{b}-\mathbf{a}]} & \text{ if } \mathbf{b} \ge \mathbf{a},\\
0& \text{ otherwise}.
\end{cases}
\]
This action is extended linearly to all of $R$ and $R^*$.  This action of $R$ on $R^*$ gives a perfect pairing of vector spaces $R_d\times R^*_d \to\F$ for any degree $d\ge 0$.  Suppose $U$ is a subspace of $R_d$.  We define
\[
U^{\perp}=\{g\in R^*_d: f\contract g=0 \mbox{ for all } f\in U\}.
\]
Macaulay \cite{Macaulay} introduced the {\em inverse system} of an ideal $I$ of $R$ to be
\[
I^{-1}:= \mbox{Ann}_{R^*}(I)=\{g\in R^*: f\contract g=0 \mbox{ for all } f\in I\}.
\]
If $I$ is a homogeneous ideal of $R$ then the inverse system $I^{-1}$ can be constructed degree by degree using the identification $(I^{-1})_d=I_d^{\perp}$. We return to this notion in \cref{defn: inverse system}.

A priori, $R^*$ is simply a graded $R$-module.  However, $R^*$ can be equipped with a multiplication which makes it into a ring.  Suppose $\mathbf{a}=(a_1,\ldots,a_n),\mathbf{b}=(b_1,\ldots,b_n)\in \Z^{n}_{\ge 0}$.  The multiplication in $R^*$ is defined on monomials by
\begin{equation}\label{eq:dividedmultiplication}
X^{[\mathbf{a}]}X^{[\mathbf{b}]}= {\mathbf{a}+\mathbf{b} \choose \mathbf{a} } X^{[\mathbf{a+b}]},
\end{equation}
where
\begin{equation}\label{eq:multifactorial}
\mathbf{a}!=\prod_{i=0}^N a_i ! \quad\mbox{and}\quad \binom{\mathbf{a}+\mathbf{b}}{\mathbf{a}}=\prod_{i=1}^n \binom{a_i+b_i}{a_i}.
\end{equation}
This multiplication is extended linearly to all of $R^*$.   We see from the above definition that if $\mathbf{a}=(a_1,\ldots,a_n)$ then $X^{[\mathbf{a}]}=\prod_{i=1}^n X_i^{[a_i]}$.  

\begin{exer}\label{ex: 6.4}
Now set $X^{\mathbf{a}}=\prod_{i=1}^n X_i^{a_i}$, where the multiplication occurs in the divided power algebra as defined above. Deduce from the above definition that
\begin{equation}\label{eq:regularmonomialtodividedmonomial}
X^{\mathbf{a}}=\mathbf{a}! X^{[\mathbf{a}]}.
\end{equation}
\end{exer}

\begin{rem} In characteristic zero, $\mathbf{a}!$ never vanishes and so, by \eqref{eq:regularmonomialtodividedmonomial}, $R^*$ is generated as an algebra by $X_0,\ldots,X_N$, just like the polynomial ring.  However, in charateristic $p>0$, $R^*$ is infinitely generated by all the divided power monomials $X_j^{[p^{k_i}]}$ for all $j=0,\ldots ,N$ and $k_j\ge 0$. The exercise below justifies this last assertion.
\end{rem}
 
 \begin{exer}
 Prove that in characteristic $p$ for any $\mathbf{a}=(a_1,\ldots,a_n)$ where $a_j=\sum a_{ij}p^i$, we have
\[
X^{[\mathbf{a}]}=\prod_{j=1}^n \prod_{i} (X_j^{[p^i]})^{a_{ij}}.
\]
{\em Hint:} Use Lucas' identity -- given base $p$ expansions $a=\sum a_ip^i$ and $b=\sum b_ip^i$ for $a, b\in\N$,  then
\[
{b \choose a} = \prod_{i=0}^\infty {b_i \choose a_i} \mod p.
\]
\end{exer}

We now revisit the characteristic zero case.  Suppose $\F$ is a field of characteristic zero and let $S=\F[X_1,\ldots,X_n]$ be a polynomial ring.  Consider the action of $R$ on $S$ by partial differentiation, which we represent by `$\circ$'.  That is, if $\mathbf{a}=(a_1,\ldots,a_n)\in \Z^{n}_{\ge 0}$, $x^\mathbf{a}=x_1^{a_0}\cdots x_n^{a_n}$ is a monomial in $R$, and $g\in S$, we write
\[
x^{\mathbf{a}}\circ g=\frac{\partial^{\mathbf{a}}g}{\partial X^{\mathbf{a}}}
\]
for the action of $x^{\mathbf{a}}$ on $g$ (extended linearly to all of $R$).  In particular, if $\mathbf{a}\le \mathbf{b}$, then
\[
x^{\mathbf{a}}\circ X^{\mathbf{b}}=\frac{\mathbf{b}!}{(\mathbf{b}-\mathbf{a})!}X^{\mathbf{b}-\mathbf{a}},
\]
where we  use the conventions in~\eqref{eq:multifactorial}.  This action gives a perfect pairing $R_d\times S_d\to\F$, and, given a homogeneous ideal $I\subset R$, we define $I_d^{\perp}$ and $I^{-1}$ in the same way as we do for contraction.

Since we are in characteristic zero, the map of rings $\Phi:S\to R^*$ defined by $\Phi(X_i)=X_i$ extends to all monomials via~\eqref{eq:regularmonomialtodividedmonomial} to give $\Phi(y^\mathbf{a})=Y^{\mathbf{a}}=\mathbf{a}! Y^{[\mathbf{a}]}$.  Thus $S$ and $R^*$ are isomorphic rings in view of \Cref{ex: 6.4}.  Moreover, if $F\in R$ and $g\in S$, then $\Phi(F\circ g)=F\contract\Phi(g)$~\cite[Theorem~9.5]{Ger95}, so $S$ and $R^*$ are isomorphic as $R$-modules.

\subsection{Macaulay inverse systems}\label{MIS}
\begin{defn}[Dualizing functor]
Let $M$ be a finitely generated $R$-module. Define the dual of $M$ to be
$
D(M) = \Hom_R(M, R^*)
$.
Let $f:M\to N$ be an $R$-module homomorphism. Define $D(f)$ to be the induced $R$-module homomorphism 
\[D(f):D(N)=\Hom_R(N, R^*)\to D(M)=\Hom_R(M, R^*)\] given by
\[
D(f)(\varphi)=\varphi\circ f.
\]
This makes $D$ into a contravariant functor in the category of finitely generated $R$-modules.
\end{defn}

\begin{exer}
 Let $M$ be a finitely generated $R$-module. Recall that we defined $D(M)=\Hom_R(M, R^*)$. In this exercise we also consider the set  $M^*=\Hom_\F(M,\F)$ with its two structures induced from the $R$-module structure of $M$ by setting
\[
r\phi(x)=\phi(r\cdot x), \quad \forall r\in R, x\in M.
\]
 Show that $D(M)\cong \Hom_\F(M,\F)$ as $R$-modules,  so an equivalent way to define the dual module dual to $M$ is $M^*$ (with its $R$-module structure).
 
 {\em Hint:} Hom-tensor adjointness may come in handy.
\end{exer}

We now come to a form of duality that involves the above defined functor.

\begin{thm}[Matlis duality]\label{thm:Matlis}
\label{thm: Matlis}
The functor $D$ induces an anti-equivalence of categories between
\[
\{\text{noetherian $R$-modules}\} \leftrightarrow \{\text{artinian $R$-submodules of $R^*$}\}
\]
given by sending $M\mapsto D(M)$.
\end{thm}

Next we wish to make the meaning of $D(M)$ more concrete in the special case when $M=R/I$ is a cyclic $R$-module.
\begin{lem}
\label{lem:DR/I}
Suppose $I$ is a homogeneous ideal of a polynomial ring $R$. We compute 
$$D(R/I)=\Hom_R(R/I, R^*)\cong \Ann_{R^*}(I) =(0:_{R^*} I)=\{g\in R^* \mid f \contract g=0 \ \forall f\in I\}.$$
\end{lem}

\begin{defn}[Inverse system]\label{defn: inverse system}
Suppose $I$ is a homogeneous ideal of a polynomial ring $R$. The {\em inverse system} of $I$ is the vector space
\[
I^{-1}=\{g\in R^*\mid f\contract g=0, \forall f\in I\}.
\]

\end{defn}

\begin{rem}
Don't let the notation deceive you! If $I$ is an ideal of $R$, it does not mean that $I^\perp$ is an ideal (or $R^*$-submodule) of $R^*$. It is just an $R$-module which happens to be a subset of $R^*$.
\end{rem}

\begin{ex}
Concretely, say 
\begin{enumerate}
\item $I=(x^2,y^3)\subseteq R=\F[x,y]$. Then 
\[
I^{-1}=(0:_{R^*} I)=\Span\{XY^2, XY, Y^2, X, Y, 1\}=R\contract XY^2
\]
 is the $R$-submodule of $R^*$ generated by $XY^2$.
\item $I=(x^2,xy^2,y^3)\subseteq R=\F[x,y]$. Then 
\[
I^{-1}=(0:_{R^*} I)=\Span\{XY, Y^2, X, Y, 1\}=R\contract XY +R\contract Y^2
\] is an $R$-submodule of $R^*$ with two generators.
\end{enumerate}
 \end{ex}
 
 Next, take 
 
 \begin{ex}\label{ex: perps}
 \begin{enumerate}
 \item $I=(x)\subseteq R=\F[x,y]$. Then 
 \[
 I^{-1}=(0:_{R^*} I)=\Span\{Y^i\mid i\geq 0\}.
 \]
\item $I=(x^d)\subseteq R=\F[x,y]$. Then 
\begin{eqnarray*}
I^{-1}=(0:_{R^*} I) &=& \Span\{X^iY^j\mid 0\leq i\leq d-1, j\geq 0\} \\
&=&
R^*_0\oplus R^*_1 \oplus R^*_2 \oplus \cdots \oplus R^*_{d-1} \oplus YR^*_{d-1} \oplus Y^2R^*_{d-1}\oplus \cdots \oplus Y^kR^*_{d-1}\oplus \cdots
\end{eqnarray*}
\end{enumerate}
\end{ex}

Both of the above $(0:_{R^*} I)$ are non-finitely generated $R$-modules. \Cref{thm:Matlis} shows that this corresponds to $R/I$ not being artinian.

\begin{exer}
Generalize \cref{ex: perps} to find the inverse system of the ideal defining a point in projective $n$-space and the inverse systems of all of the powers of this ideal.
\end{exer}

We now wish to study the inverse functor involved in the Matlis duality \cref{thm: Matlis}. In order to do this we define the inverse system of an $\F$-subspace of $R^*$.

\begin{defn}
Let $V$ be an $\F$-vector subspace of the $\F$-algebra $R^*$. The {\em inverse system} of $V$ is
\[
V^\perp=\Ann_R(V)=\{f\in R\mid f\circ v=0, \forall v\in V\}.
\]
We will be most interested in the case when $V=\Span\{F\}$ is a 1-dimensional $\F$-vector space and thus
\[
F^\perp=\Ann_R(F)=\{f\in R\mid f\circ F=0\}.
\]
\end{defn}

Macaulay inverse system duality is a concrete version of Matlis duality \cref{thm: Matlis} which can be stated in terms of the inverse systems defined above as follows:

\begin{thm}[Macaulay inverse system duality]
\label{thm:Mac}
With notation as above, there are bijective correspondence between 
\begin{eqnarray*}
\{R-\text{modules } M\subseteq R^*\} &\leftrightarrow & \{R/I \mid  I\subseteq R \text{ homogeneous ideal}\} \\
M &\mapsto & R/\Ann_R(M)\\
I^\perp=D(R/I) &\mapsfrom& R/I.
\end{eqnarray*}

Furthermore, we have the additional correspondences

\begin{center}
\begin{tabular}{cccc}
(a) & $M$ finitely generated & $\iff$  & $R/\Ann_R(M)$ artinian\\
(b) & $M=R\circ F$ cyclic& $\iff$  & $R/\Ann_R(F)$ artinian Gorenstein \\
 & & & $\deg(F)=$ socle degree of $R/\Ann_R(F)$.
\end{tabular}
\end{center}
\end{thm}

The value of  \cref{thm:Mac} often lies in producing examples of artinian Gorenstein rings.

\begin{defn}
In view of statement (b) in \cref{thm:Mac}, the polynomial $F\in R^*$ is called a {\em Macaulay dual generator} for $R/\Ann_R(F)$. 
\end{defn}

\begin{ex}
The artinian Gorenstein algebra with Macaulay dual generator $$F=X^2+Y^2+Z^2$$ is the ring of \cref{ex: Gor}
$$\F[x,y,z]/\Ann_{\F[x,y,z]}(F)=\F[x,y,z]/(x^2-y^2,y^2-z^2,xy,xz,yz).$$
\end{ex}

\begin{ex}
The artinian Gorenstein algebra with Macaulay dual generator $$F=X_1^{d_1}\cdots X_n^{d_n}$$ is the monomial complete intersection
$$\F[x_1,\ldots, x_n]/\Ann_{\F[x_1,\ldots,x_n]}(F)=\F[x_1,\ldots, x_n]/(x_1^{d_1+1},\ldots,x_n^{d_n+1}).$$
\end{ex}

\begin{defn}
For a graded module $M$ and an integer $d$, define $M(d)$ to be the module $M$ with grading modified such that $M(d)_i=M_{d+i}$.
\end{defn}

\begin{exer} \label{ex: Gor}
For any homogeneous polynomial $F\in R^*$ of degree $d$, prove
\begin{enumerate}
\item $\Ann_R(F)^\perp=R\circ F$. This statement is an instance of Macaulay's double annihilator theorem.
\item  the cyclic ring  $A=R/\Ann_R(F)$ is artinian Gorenstein if and only if  the function $A\to D(A)(-d)$, $a\mapsto [b\mapsto (ab)\circ F]$ is an isomorphism.
\end{enumerate}
 {\em Hint for (1):} Start by showing the equality is true in degree $d$, then use the $R$-module structure.\\ {\em Hint for (2):} Use \cref{prop: Poincare}. Prove that the function $a \mapsto (a\contract F)(0)$ is an orientation on $A$ and that $A$ satisfies Poincar\'e duality with respect to this orientation.
\end{exer}

\noindent In view of \cref{ex: Gor} we can state an alternate definition of graded Gorenstein rings.

\begin{defn}
An artinian  graded ring $A$ is  {\em Gorenstein} of socle degree $d$ if and only if $A\cong D(A)(-d)$ as graded $A$-modules (degree preserving isomorphism).
\end{defn}

\subsection{SLP for Gorenstein rings via Hessian matrices}\label{HESS}

For this section let  $R=\F[x_1,\dots,x_n]$ be a polynomial ring and $R^*$ its graded dual. We will further assume that $\chr(\F)=0$.

In this section we use that $R^*$ is isomorphic to $\F[X_1,\ldots, X_n]$  with $R$-action $x_i\circ F=\frac{\partial F}{\partial X_i}$. We will use this description for $R^*$.

\begin{lem}\label{lem: 6.24}
Let $F\in R^*_c$ and let $L=a_1x_1+\dots+a_nx_n\in R_1$. Then $$L^c\circ F=c!\cdot F(a_1,\ldots, a_n).$$
\end{lem}
\begin{proof}
\[
L^c\circ F=\sum_{i_1+\cdots+i_n=c}\frac{c!}{i_1!\cdots i_n!}a_1^{i_1}\cdots a_n^{i_n} x_1^{i_1}\cdots x_n^{i_n}\circ F=c!\cdot F(a_1,\ldots, a_n).
\]
\end{proof}

\begin{defn}[Higher Hessians] Let $F\in R^*$ be a homogeneous polynomial and let $B=\{b_1,\ldots, b_s\}\subseteq R_d$ be a finite set of homogeneous polynomials of degree $d \geq 0$. We call 
\[
\Hess_B^d(F)=\left[b_ib_j \circ F\right]_{1\leq i,j\leq s} \text{ and } \hess_B^d(F)=\det \Hess_B^d(F)
\]
the $d$-th {\em Hessian matrix}  and the $d$-th {\em Hessian determinant} of $F$ with respect to $B$, respectively.
\end{defn}

\begin{rem}
If $B=\{x_1,\ldots x_n\}$ then $\Hess_B^1(F)=\left[x_ix_j F\right]_{1\leq i,j\leq n}=\left[\frac{\partial F}{\partial X_i \partial X_j}\right]_{1\leq i,j\leq n}$ is the classical Hessian of $F$.
\end{rem}

Hessians are useful in establishing the SLP for artinian Gorenstein rings.

\begin{thm}[Hessian criterion for SLP \cite{MW}]\label{thm: hessian criterion}
Assume $\F$ is a field of characteristic zero. Let $A$ be a graded artinian Gorenstein ring with Macaulay dual generator $F\in R^*_c$. Then $A$ has the SLP if and only if 
\[ \hess^i_{B_i}(F)\neq 0 \text{ for } 0\leq i\leq \lfloor \frac{c}{2} \rfloor \]
where $B_i$ is some (equivalently, any) basis of $A_i$.
\end{thm}

\begin{proof}
From the hypothesis and \cref{thm:Mac}
we have that $A=R/\Ann_R(F)$ has socle degree $d=\deg(F)$.

Since $A$ is Gorenstein, $A$ has symmetric Hilbert function, so by \Cref{equiv narrow SLP} $A$ has SLP if and only if $A$ has $SLP$ in the narrow sense, i.e. there exists $L\in A_1$ such that for any $0\leq i\leq \lfloor \frac{d}{2}\rfloor$ the multiplication maps $L^{c-2i}:A_i\to A_{c-i}$ are vector space isomorphisms. Say $L=a_1x_1+\dots+a_nx_n$.

Recall that the isomorphism $A\cong D(A)(-c)=(R\circ F)(-c), a\mapsto a\circ F$ induces vector space isomorphisms $A_{c-i}\cong A_{i}^*$ also defined by $a\mapsto \left[b\mapsto b\circ (a\circ F)=(ba)\circ F\right]$; see \Cref{ex: Gor}. This isomorphism is denoted $- \circ F$ in the sequence displayed below. The composite map
\[
T_i:A_i\stackrel{\times L^{c-2i}}{\lra} A_{c-i} \stackrel{-\circ F}{\lra} A_i^*
\]
is an isomorphism if and only if multiplication by $L^{c-2i}$ is an isomorphism. Let $B_i$ be a basis for $A_i$ and let $B^*_i$ be its dual, which is a basis for $A_i^*$. The matrix $[t^{(i)}_{jk}] $for $T_i$ with respect to these bases is defined as follows
\[
T_i(b_j)=\sum_{k=1}^s t^{(i)}_{jk}b_k^*.
\]
Hence  we compute
\[
t^{(i)}_{jk}=T_i(b_j)(b_k)=(L^{c-2i}b_j)^* (b_k)=(L^{c-2i}b_j b_k) \circ F=L^{c-2i}\circ(b_j b_k \circ F).
\] 
Using \Cref{lem: 6.24}  the formula above becomes
\[
t^{(i)}_{jk}=(c-2i)!(b_jb_k\circ F)(a_1,\ldots, a_n).
\]
Thus $T_i$ is an isomorphism for some $L\in R_1$ if and only if  
$$\hess^i_{B_i}F(a_1,\ldots, a_n)=\det\left[b_ib_j \circ F(a_1,\ldots, a_n)\right]_{1\leq i,j\leq s}\neq 0.$$

Overall the SLP holds if and only if for  $ 0\leq i\leq \lfloor \frac{c}{2} \rfloor$ the hessian determinant $\hess^i_{B_i}F$ does not vanish identically.
\end{proof}

\begin{ex}
Say $F=X^{2}+Y^{2}+Z^{2}$. Then with respect to the standard monomial basis for each $R_i$
\begin{eqnarray*}
\hess^0(F) &=&F \\
\hess^1(F) &=&\det \begin{bmatrix} 2 & 0 & 0 \\ 0 & 2 &0 \\ 0 & 0 & 2\end{bmatrix} =8\\
\hess^i(F) &=&0 \text{ for }i\geq 2.
\end{eqnarray*}
This shows the algebra in \cref{exer: Gor} has SLP in characteristic zero.
\end{ex}

\begin{ex}[H. Ikeda \cite{Ikeda}]
\label{ex:Ikeda}
Let $G=XYW^3+X^3ZW+Y^3Z^2$. Then $A=R/
\Ann_R(G)$ has Hilbert function $(1, 4, 10, 10, 4, 1)$ and a basis for $A_1$ is $B_1=\{x, y, z,w\}$ whereas a basis for $A_2$ is $B_2=\{x^2, xy, xz, xw, y^2, yz, yw, z^2, zw, w^2\}$. Furthermore
\begin{eqnarray*}
\hess^0(G) &=&G \\
\hess^1_{B_1}(G) &=&\det \begin{pmatrix}
     6XZW & W^3 & 3X^2W & 3X^2Z+3YW^2\\
     W^3 & 6YZ^2 & 6Y^2Z &3XW^2\\
     3X^2W & 6Y^2Z & 2Y^3 & X^3\\
     3X^2Z+3YW^2 & 3XW^2 &X^3 & 6XYW
     \end{pmatrix}\neq 0 \\
  \hess^2_{B_2}(G) &=&\det \left(\begin{matrix}
   0&0&6\,W&6\,z&0&0&0&0&6\,X&0\\
  0&0&0&0&0&0&0&0&0&6\,W\\
   6\,W&0&0&6\,X&0&0&0&0&0&0\\
   6\,Z&0&6\,X&0&0&0&6\,W&0&0&6\,Y\\
  0&0&0&0&0&12\,Z&0&12\,Y&0&0\\
  0&0&0&0&12\,Z&12\,Y&0&0&0&0\\
   0&0&0&6\,W&0&0&0&0&0&6\,X\\
   0&0&0&0&12\,Y&0&0&0&0&0\\
  6\,X&0&0&0&0&0&0&0&0&0\\
    0&6\,W&0&6\,Y&0&0&6\,X&0&0&0\\
     \end{matrix}\right) 
=0.
\end{eqnarray*}
We conclude that the map $L:A_2\to A_3$ fails to have maximum rank for all $L\in A_1$. However the map $L^3:A_1\to A_4$ does have maximum rank.
\end{ex}

\begin{exer}[R. Gondim \cite{Gondim}]\label{ex: Gondim}
Let $x_1,\ldots, x_n$ and $u_1, \ldots, u_m$ be two sets of indeterminates
with $n \geq m \geq 2$. Let $f_i \in \F[x_1, \ldots, x_n]_k$ and $g_i \in \F[u_1,\ldots, u_m]_e$ for $1\leq i\leq s$
be linearly independent forms with $1 \leq k<e$. If $s >\binom{m-1+k}{k}$, then 
\[
F=f_1g_1+\cdots+f_sg_s
\]
is called a Perazzo form and $A=\F[x_1, \ldots, x_n,u_1,\ldots,u_m]$ is called a Perazzo algebra. 
\begin{enumerate}
\item  Show that  $\hess_k(F)=0$ and so $A$ does not have SLP.
\item Make conjectures regarding the Hilbert functions of Perazzo algebras. 
\item Make conjectures regarding the WLP for  Perazzo algebras. 
\item Do there exist two Perazzo algebras $A$ and $B$ having the same Hilbert function so that $A$ has WLP and $B$ does not?
\end{enumerate}
Some answers to (2) and (3) can be found in  \cite{AADFMMMN}. Part (4) is an open problem suggested by Lisa Nicklasson.
\end{exer}

\begin{cor}
Let $F\in \F[X_1,\ldots, X_n], G\in \F[Y_1,\ldots, Y_m]$ be homogeneous polynomials of the same degree. Then $A=\F[x_1,\ldots, x_n]/\Ann_R(F)$ and $B= \F[y_1,\ldots, y_m]/\Ann_R(G)$ have SLP if and only if 
\[
C= \F[x_1,\ldots, x_n,y_1,\ldots, y_m]/\Ann_R(F+G) \text{ satisfies SLP}.
\]
\end{cor}
\begin{proof}
It turns out that for $1\leq i< \deg(F)$ a basis $\beta$ of $C_i$ is given by the union of a basis $\beta'$ of $A_i$ and a basis $\beta''$ of $B_i$ (for a proof of this refer to \cref{prop:ConnSumF} and \cref{exactFP}) and hence the hessians of $F+G$ look like
\begin{eqnarray*}
\Hess^i(F+G) &=& \begin{bmatrix} b_ib_j(F+G)\end{bmatrix}_{b_i,b_j\in \beta}=\begin{bmatrix} b'_ib'_j(F) & 0\\ 0 & b''_ib''_j(F) \end{bmatrix}_{b'_i,b'_j\in \beta',b''_i,b''_j\in \beta''}\\
& =&\begin{bmatrix} \Hess^i(F) & 0 \\ 0 & \Hess^i(G)\end{bmatrix}\\
\hess^i(F+G) &=&\hess^i(F)\hess^i(G).
\end{eqnarray*}
Now we see that $\hess^i(F+G)\neq 0$ if and only if $\hess^i(F)\neq 0$ and $\hess^i(G)\neq 0$, which gives the desired conclusion.
\end{proof}

We take the preceding example up again in the following section, generalizing the construction that produces $C$ from $A$ and $B$ in \Cref{Def_CS}.

\section{Topological ring constructions  and the Lefschetz properties} \label{s: 7}

We have seen in \cref{s: 1} that the Lefschetz properties emerged from algebraic topology. Now we return to this idea implementing some constructions that originate in topology at the ring level. The material in \cref{s:FP} is taken from \cite{IMS} and the material in  \cref{s: BUG} is taken from \cite{IMMSW}.

\subsection{Fiber products and connected sums}\label{s:FP}

We first consider the operation termed connected sum. A connected sum of manifolds along a disc is obtained by identifying a disk in each (with opposite orientations). One can more generally take connected sums by identifying two homeomorphic sub-manifolds, one from each summand. If the  cohomology rings of the two summands  are $A$ and $B$ and the cohomology ring of the common submanifold is $T$, then it turns out that the cohomology ring of the connected sum is $A\#_TB$, a ring  that we term the connected sum of  $A$ and $B$ over $T$ in \cref{Def_CS}.

To define a connected sum of rings we need a preliminary construction. Recall that an {\em oriented} AG algebra is a pair $(A,\int_A)$ with $A$ an AG algebra and $\int_A$ an orientation as in \cref{prop: Poincare}.
A choice of orientation on $A$ also corresponds to a choice of Macaulay dual generator.

\begin{exer}
 Every orientation on $A$ can be written as the function $\int_A:A\to K$ defined by $\int_A g =(g\circ F)(0)$ for some Macaulay dual generator $F$ of $A$. The notation $(g\circ F)(0)$ refers to evaluating the element $g\circ F$ of $R'$ at $X_1=\cdots=X_n=0$.
\end{exer}

Next we discuss how the orientations of two AG algebras relate.

\begin{defn}[Thom class]
\label{def:ThomClass}
 Let $(A,\int_A)$ and $(T,\int_T)$ be two oriented AG $K$-algebras with socle degree $d$ for $A$ and $k$ for $T$, respectively, with $d\geq k$. Let $\pi: A \to T$ be a graded map. By \cite[Lemma 2.1]{IMS}, there exists a unique homogeneous element $\tau_A \in A_{d-k}$  such that $\int_A(\tau _A a)=\int_T(\pi (a))$ for all $a\in A$ ; we call it the {\em Thom class} for $\pi : A \to  T$. 
\end{defn}

Note that the  Thom class for $\pi : A \to  T$ depends not only on the map $\pi $, but also on the orientations chosen for $A$ and $T$.

\begin{ex}
 Let $(A,\int_A)$ be an oriented AG $K$-algebra with socle degree $d$. Consider $(K,\int_K)$ where $f_K:K\to K$ is the identity map. Then the Thom class for the canonical projection $\pi:A\to K$ is the unique element $a_{soc}\in A_d$ such that $\int_A a_{soc}=1$.  
\end{ex}

\begin{exer}
Given a homomorphism  $\pi: A\to T$ of AG algebras having dual generators $F, H$ of degrees  $d$ and $k$, respectively, with $d\geq k$, show that the Thom class of \cref{def:ThomClass} is the unique element $\tau$ of $A_{d-k}$ such that $\tau\circ F=H$.
\end{exer}

\begin{defn}
\label{def:fiberProduct}
	Given graded $\F$-algebras $A$, $B$, and $T$, and graded $\F$-algebra maps $\pi_A\colon A\rightarrow T$ and $\pi_B\colon B\rightarrow T$, the \emph{fiber product} of $A$ and $B$ over $T$ (with respect to $\pi_A$ and $\pi_B$) is the graded $\F$-subalgebra of $A\oplus B$
	$$A\times_T B=\left\{(a,b)\in A\oplus B \ \left| \ \pi_A(a)=\pi_B(b)\right.\right\}.$$
\end{defn}  
Let $\rho_1\colon A\times_TB\rightarrow A$ and $\rho_2\colon A\times_T B\rightarrow B$ be the natural projection maps.  It is well known that fiber products are pullbacks in the category of $\F$ algebras and hence they satisfy the following 
universal property.
\begin{lem}
	\label{lem:UnivProp}
	The fiber product $A\times_TB$ satisfies the following universal property:  If $C$ is another $\F$-algebra with maps $\phi_1\colon C\rightarrow A$ and $\phi_2\colon C\rightarrow B$ such that $\pi_A\circ\phi_1(c)=\pi_B\circ\phi_2(c)$ for all $c\in C$, then there is a unique $\F$-algebra homomorphism $\Phi\colon C\rightarrow A\times_TB$ which makes the diagram below commute:
	\begin{equation}
	\label{eq:UP}
	\xymatrix{C \ar@{-->}[dr]^{\Phi} \ar@/^1pc/[drr]^-{\phi_1}\ar@/_1pc/[ddr]_-{\phi_2}& & \\ & A\times_TB\ar[r]^-{\rho_1}\ar[d]_-{\rho_2} & A\ar[d]^-{\pi_A}\\ & B\ar[r]_-{\pi_B} & T.\\}
	\end{equation} 
\end{lem}

By \cite[Lemma 3.7]{IMS} the fiber product is  characterized by the following exact sequence of vector spaces:
\begin{equation}\label{exactFP}
    0 \to    A \times _T B \to A\oplus B \xrightarrow{\pi_A-\pi_B} T\to 0,
\end{equation}
whence the Hilbert function of the fiber product satisfies
\begin{equation}\label{eq: HF FP}
H_{ A \times _T B}=H_A+H_B-H_T.
\end{equation}

Henceforth we assume that $\pi_A(\tau_A) = \pi_B(\tau_B)$, so that $(\tau_A,\tau_B) \in A \times _T B$.

\begin{defn}\label{Def_CS}
    The {\em connected sum} of the oriented AG $K$-algebras $A$ and $B$ over $T$ is the quotient ring of the fiber product 
    \[A\times _T B
    := \{(a, b) \in A\oplus B \mid  \pi_A(a) = \pi _B(b) \}\]
    by the principal
ideal generated by the pair of Thom classes 
$(\tau_A, \tau_B)$, i.e.
$$ A \#_TB =
(A \times_T B)/
\langle (\tau_A, \tau_B) \rangle .$$
\end{defn}

By \cite[Lemma 3.7]{IMS} the connected sum is  characterized by the following exact sequence of vector spaces:
\begin{equation}\label{exactCS}
    0 \to  T(k - d) \to  A \times _T B \to  A\# _T B \to  0.
\end{equation}

Therefore, the Hilbert series of the connected sum satisfies
\begin{equation}\label{HilbertCS}
    HF_{A\# _T B}( t) = HF_A( t) + HF_B( t)- (1 + t^{d-k})HF_T( t).
\end{equation}

When $T=\F$ we have an easy description of the fiber product and connected sum.

\begin{prop}\label{prop:ConnSumF}
	Let $R=\F[x_1,\ldots, x_n]$, $R'=\F[y_1,\ldots, y_m]$ be  polynomial rings.  Let $\left(A=R/I,\int_A\right)$ and $\left(B=R'/I',\int_B\right)$ be oriented AG algebras each with socle degree $d$, and let $\pi_A\colon A\rightarrow\F$ and $\pi_B\colon B\rightarrow \F$ be the natural projection maps with Thom classes $\tau_A\in A_d$ and $\tau_B\in B_d$.  Then the fiber product $A\times_\F B$ has a presentation
	$$A\times_\F B\cong \frac{\F[x_1,\ldots,x_n,y_1,\ldots, y_m]}{(x_iy_j\mid 1\leq i\leq n, 1\leq j\leq m)+I+I'}$$
	and the connected sum $A\#_\F B$  has a presentation
	$$A\#_\F B\cong \frac{\F[x_1,\ldots,x_n,y_1,\ldots, y_m]}{(x_iy_j\mid 1\leq i\leq n, 1\leq j\leq m)+I+I'+(\tau_A+\tau_B)}.$$
	In particular, if $A$ and $B$ are standard graded then so are $A\times_\F B$ and $A\#_\F B$.
\end{prop}

\begin{ex}[Standard graded fiber product and connected sum]\label{ex:FPEx}
	Let $A=\F[x,y]/(x^2,y^4)$ and $B=\F[u,v]/(u^3,v^3)$ each with the standard grading $\deg(x)=\deg(y)=\deg(u)=\deg(v)=1$.  Let $T=\F[z]/(z^2)$, and define maps $\pi_A:A\rightarrow T$, $\pi_A(x)=z, \ \pi_A(y)=0$ and $\pi_B\colon B\rightarrow T$, $\pi_B(u)=z, \ \pi_B(v)=0$.  Then the fiber product $A\times_TB$ is generated as an algebra by elements $z_1=(y,0)$, $z_2=(x,u)$, and $z_3=(0,v)$, all having degree one.  One can check that it has the following presentation:
	\begin{equation}
	\label{eq:FPEx}
	A\times_TB=\frac{\F[z_1,z_2,z_3]}{\left\langle z_1^4,z_2^3,z_3^3, z_1z_3,z_1z_2^2\right\rangle}.
	\end{equation}
	The Hilbert function of the fiber product is 
	\begin{align*}
	H(A\times_TB)= & (1,3,5,4,2)\\
	= & (1,2,2,2,1)+(1,2,3,2,1)-(1,1,0,0,0)\\
	= & H(A)+H(B)-H(T).
	\end{align*}
	
	Fix orientations on $A$, $B$, and $T$ by $\int_A\colon xy^3\mapsto 1$, $\int_B\colon u^2v^2\mapsto 1$, and $\int_T\colon z\mapsto 1$, respectively.  Then the Thom classes for $\pi_A\colon A\rightarrow T$ and $\pi_B\colon B\rightarrow T$ are, respectively, $\tau_A=y^3$, $\tau_B=uv^2$.  Note that $\pi_A(\tau_A)=0=\pi_B(\tau_B)$, hence $(\tau_A,\tau_B)\in A\times_TB$, and in terms of Presentation \eqref{eq:FPEx} we have $(\tau_A,\tau_B)=z_1^3+z_2z_3^2$.  Therefore we see that 
	\begin{equation}
	\label{eq:CSEx}
	A\#_TB=\frac{\F[z_1,z_2,z_3]}{\left\langle z_1^4,z_2^3,z_3^3, z_1z_3,z_1z_2^2,z_1^3+z_2z_3^2\right\rangle}.
	\end{equation}
	The Hilbert function of the connected sum is 
	\begin{align*}
	H(A\#_TB)= & (1,3,5,3,1)\\
	= & (1,2,2,2,1)+(1,2,3,2,1)-(1,1,0,0,0)-(0,0,0,1,1)\\
	= & H(A)+H(B)-H(T)-H(T)[3],
	\end{align*}
	where $H(T)[3]$ is the Hilbert function of $T(-3)$.
\end{ex}

However, if $T\neq \F$ the presentation of the connected sum and fiber product can be complicated and they need not be standard graded.

\begin{ex}[Non-standard graded fiber product and connected sum]
\label{ex:FP3}
Let 
$$A=\F[x]/(x^4), \ B=\F[u,v]/(u^3,v^2), \ T=\F[z]/(z^2),$$  have Hilbert functions $H(A)=(1,1,1,1)$ and $H(B)=(1,2,2,1)$.
Define maps $\pi_A\colon A\rightarrow T$, $\pi_A(x)=z$ and $\pi_B\colon B\rightarrow T$, $\pi_B(u)=z$, $\pi_B(v)=0$.  Then the fibered product has the presentation
$$A\times_TB=\frac{\F[z_1,z_2,z_3]}{\left(z_1^4,z_2^2,z_3^2,z_1z_3,z_1^2z_2-z_2z_3\right)}, \ \ \text{where} \begin{cases} z_1= & (x,u)\\
z_2= & (0,v)\\
z_3= & (0,u^2).
\end{cases}$$
 Here $z_1,z_2$ have degree one, and $z_3$ has degree two. 
 We then have a presentation for the connected sum $C=A\#_TB=A\times_T B/(\tau)$,
whence
\begin{align*}
A\#_TB \cong & \frac{\F[z_1,z_2,z_3]}{\left(z_1^4,z_2^2,z_3^2,z_1z_3,z_1^2z_2-z_2z_3,(z_1^2-z_3)+z_1z_2\right)}\cong \frac{\F[z_1,z_2]}{(z_1^3+z_1^2z_2,z_2^2)}.
\end{align*}
It has Hilbert function $H(C)=(1,2,2,1)=H(A)+H(B)-H(T)-H(T)[1]$ as in \eqref{HilbertCS}.  It is interesting to note that the connected sum $A\#_TB$ has a standard grading whereas the fibered product $A\times_TB$ does not.  
\end{ex}

Finally, we have the following result which shows how the Lefschetz properties of the components influence the Lefschetz property of the fiber product and connected sum.

\begin{thm}
	\label{prop:SLPFP}
	\begin{enumerate}
\item If $A$ and $B$ are artinian Gorenstein algebras of the same socle degree that each have the SLP, then the fiber product $D=A\times_\F B$ over a field $\F$  also has the SLP.  If $A$ and $B$ have the standard grading, then the converse holds as well. 	
\item If $A$ and $B$ both have the SLP, then the connected sum $C=A\#_\F B$ over a field $\F$ also has the SLP.  If $A$ and $B$ have the standard grading, then the converse holds as well.
\item Let $A, T$ be artinian Gorenstein algebras with socle degrees $d, k$ respectively and let $\pi_A\colon A\rightarrow T$ be a surjective ring homomorphism such that its Thom class $\tau_A$ satisfies $\pi_A(\tau_A)=0$.  Let $x$ be an indeterminate of degree one, set $B=T[x]/(x^{d-k+1})$, and define $\pi_B\colon B\rightarrow T$ to be the natural projection map satisfying $\pi_B(t)=t$ and $\pi_B(x)=0$.  In this setup, if $A$ and $T$ both satisfy the SLP,  then the fiber product $A\times_T B$ also satisfies the SLP.  Moreover if the field $\F$ is algebraically closed, then the connected sum $A\#_T B$  also satisfies the SLP.
\item Let $A$ and $B$ be standard graded artinian Gorenstein algebras of socle degree $d$ satisfying the SLP, and let $T$ be a graded artinian Gorenstein algebra of socle degree $k$, with $k<\lfloor \frac{d-1}{2} \rfloor$,
endowed with surjective $\F$-algebra homomorphisms $\pi_A:A\to T$ and $\pi_B:B\to T$. Then the resulting fiber product  $A\times_TB$ and the connected sum $A\#_T B$ both satisfy the WLP.
\end{enumerate}
\end{thm}

\begin{ex}
 Take $$F=XY(XZ-YT)\in K[X,Y,Z,T]$$ and set $A=K[x,y,z,t]/\Ann(F)$.
   Then   
\[
\Ann(F)=(zt, xz + yt,x^{2}t,y^{2}z,x^{2}y^{2},x^{3},y^{3},z^{2},t^{2}),
\]
 $A$ is a connected sum 
    \[
    A=K[x,y,z]/\Ann(X^2YZ)\#_{K[x,y]/\Ann(XY)} K[x,y,t]/\Ann(XY^2T),
    \]
    and  the   Hilbert function of $A$ is $(1, 4, 6, 4, 1)$. 
    By \cref{prop:SLPFP} (4), since the summands of $A$ are monomial complete intersections, $A$ has WLP if the characteristic of $\F$ is 0.
\end{ex}

\begin{ex}[\cite{ADFMMSV}]
 Take $$F=X^3YZ-XY^3T=XY(X^2Z-Y^2T)\in K[X,Y,Z,T]$$ and set $A=K[x,y,z,t]/\Ann(F)$. Then 
     $$\Ann(F)=(z^2,t^2,tz,x^2t,y^2z,x^2z+y^2t,y^4,x^2y^2,x^4),$$
$A$ is a connected sum 
    \[
    A=K[x,y,z]/\Ann(X^3YZ)\#_{K[x,y]/\Ann(XY)} K[x,y,t]/\Ann(XY^3T),
    \]
    and   the   Hilbert function of $A$ is $(1, 4, 7, 7, 4, 1)$.
    
     The Hessian matrix of $F$ of order two is of the following form 
$$
\Hess^2(F)=6\begin{pmatrix}
0&y&x&z&0&0&0\\
y&0&0&x&0&0&0\\
x&0&0&0&0&0&0\\
z&x&0&0&0&-y&-t\\
0&0&0&0&0&0&-y\\
0&0&0&-y&0&0&-x\\
0&0&0&-t&-y&-x&0\\
\end{pmatrix}
$$
and it has vanishing determinant. According to  the Hessian criteria \cref{thm: hessian criterion} $A$ does not have WLP because in this case the second Hessian corresponds to the multiplication map from degree 2 to degree 3. Note that the socle degrees don't satisfy the condition in \cref{prop:SLPFP} since $2=k=\left \lfloor\frac{d-1}{2} \right\rfloor=\frac{5-1}{2}$.
\end{ex}

\subsection{Cohomological blowups}\label{s: BUG}

The second construction is inspired by the geometric operation of blowing up a smooth projective algebraic variety. The blow-up of such a space at a point replaces the point with the set of all directions through the point, that is, a projective space. More generally one can blow up a subset and replace it with another space called an {\em exceptional divisor}. The cohomology ring of the blow-up can be determined based on the cohomology ring of the original variety (called $A$ below), that of the subvariety being blown up (called $T$ below)  and the way the latter sits inside the former, specifically captured via the cohomology class of the {\em normal bundle} of the subvariety, encoded via a polynomial $f_A(\xi)$ below. 

We now explain the algebraic construction for the cohomology ring of a blowup.

\begin{defn}[Cohomological Blow-Up]
	\label{def:blowup}
	For oriented artinian Gorenstein algebras $A$ and $T$ of socle degrees $d>k$, respectively, and surjective degree-preserving algebra map $\pi\colon A\rightarrow T$ with Thom class $\tau\in A_n$ where $n=d-k$, set $K=\ker(\pi)$. Given a homogeneous monic polynomial $f_A(\xi)=\xi^n+a_1\xi^{n-1}+\cdots+a_n\in A[\xi]$ of degree $n$ with homogeneous elements $a_i\in A_i$ for $1\leq i\leq n$ and with $a_n=\lambda\cdot \tau$ for some non-zero constant $\lambda$,  we call the  artinian Gorenstein algebra $\tilde{A}$ below a \emph{cohomological blow up of $A$ along $\pi$} or BUG for short
	\[
	\tilde{A}=\frac{A[\xi]}{(\xi \cdot K,\underbrace{\xi^n+a_1\xi^{n-1}+\cdots+\lambda\cdot\tau}_{f_A(\xi)})}.
\]
	 Setting $t_i=\pi(a_i)$ for $1\leq i\leq n-1$, the AG algebra 
	 \[
	 \tilde{T}=\frac{T[\xi]}{(\underbrace{\xi^n+t_1\xi^{n-1}+\cdots+\lambda\cdot \pi(\tau))}_{f_T(\xi)}}
	 \]
	is called the \emph{exceptional divisor of $T$ with parameters $(t_1,\ldots,t_{n-1},\lambda)$}.  These algebras fit in the following commutative diagram, where  we refer to $A$ as the \emph{cohomological blow down of $\tilde{A}$ along $\hat{\pi}$}.
	\begin{equation*}
\label{eq:cd}
\xymatrix{A\ar[r]^-{}\ar[d]_-{\pi} & \tilde{A}\ar[d]^-{\hat{\pi}}\\ T\ar[r]_-{} & \tilde{T}.}
\end{equation*}  
\end{defn}

Since $\tilde{T}$ is a quotient of a 1-dimensional Gorenstein ring by a non zero-divisor, it is clear that  $\tilde{T}$ is artinian Gorenstein. It is shown in \cite{IMMSW} that the condition that the last term of $f_A(\xi)$ be a scalar multiple of the Thom class $\tau$ is precisely equivalent to  $\tilde{A}$ being AG.

\begin{ex}
\label{ex:notGor}
Let 
$$A=\frac{\F[x,y]}{(x^3,y^3)} \ \overset{\pi}{\rightarrow} \ T=\frac{\F[x,y]}{(x^2,y)}$$
where $\pi(x)=x$ and $\pi(y)=0$.  Note $K=\ker(\pi)=(x^2,y)$.  Orient $A$ and $T$ with socle generators $a_{soc}=x^2y^2$ and $t_{soc}=x$; then the Thom class of $\pi$ is $\tau=xy^2\in A_3$.  Set $f_T(\xi)=\xi^3+x\xi^2\in T[\xi]$ and let $\tilde{T}$ be the associated exceptional divisor algebra:
$$\tilde{T}=\frac{T[\xi]}{(f_T(\xi))}=\frac{\F[x,y,\xi]}{(x^2,y,\xi^3+x\xi^2)}.$$
Consider $f_A(\xi)=\xi^3+x\xi^2+xy^2\in A[\xi]$. This gives rise to the BUG 
$$\tilde{A}=\frac{A[\xi]}{(\xi \cdot K,f_A(\xi))}=\frac{\F[x,y,\xi]}{(x^3,y^3,x^2\xi,y\xi,\xi^3+x\xi^2+xy^2)}$$ 
which has basis 
$$\left\{1,x,y,\xi,x^2,xy,y^2,x\xi,\xi^2,x^2y,xy^2,x\xi^2,x^2y^2\right\}$$
and Hilbert function 
$H(\tilde{A})=(1,3,5,3,1).$
Here the socle of $\tilde{A}$ is generated by $\tilde{a}_{soc}=a_{soc}=x^2y^2$, hence $\tilde{A}$ is Gorenstein, as expected.  
\end{ex}

 We are now ready to discuss the Lefschetz properties for cohomological blow-up algebras.

\begin{thm}[{\cite[Theorem 8.5]{IMMSW}}]
\label{thm:SLPblowup}
Let $\F$ be an infinite field and let $\pi:A\to T$ be a surjective homomorphism of  graded AG $\F$-algebras of socle degrees $d>k$ respectively such that both $A$ and $T$ have SLP.  Assume that characteristic $\F$ is zero or characteristic $F$ is $ p>d$. Then every cohomological blow-up algebra of $A$ along $T$ satisfies SLP.
\end{thm}

The following example shows that the converse of \cref{thm:SLPblowup} is not true: if the cohomological blowup $\tilde{A}$ has SLP it does not follow that $A$ has SLP. In other words, while the process of blowing up preserves SLP, the process of blowing down does not preserve SLP, nor even WLP.

\begin{ex}
	\label{ex:Rodrigo}
As in \cref{ex: Gondim}, the following example, originally due to U. Perazzo \cite{Perazzo}, but re-examined more recently by R. Gondim and F. Russo \cite{GR}, is an artinian Gorenstein algebra with unimodal Hilbert function which does not have SLP or WLP:
	\begin{eqnarray*}
	A &=& \frac{\F[x,y,z,u,v]}{\operatorname{Ann}(XU^2+YUV+ZV^2)} \\
	&=&\frac{\F[x,y,z,u,v]}{\left(x^2,xy,y^2,xz,yz,z^2,u^3,u^2v,uv^2,v^3,xv,zu,xu-yv,zv-yu\right)}.
	\end{eqnarray*}
	Taking the quotient $T$ of $A$ given by the Thom class $\tau=u^2$ yields
	$$T=\frac{\F[x,y,z,u,v]}{\operatorname{Ann}(X)}=\frac{\F[x,y,z,u,v]}{\left(x^2,y,z,u,v\right)}\cong\frac{\F[x]}{(x^2)}.$$
	Fix a parameter $\lambda\in\F$ and define  polynomials $f_T(\xi)\in T[\xi]$ and $f_A(\xi)\in A[\xi]$ by
	$$f_T(\xi)=\xi^2-\lambda x\xi \quad \text{ and } f_A(\xi)=\xi^2-\lambda x\xi+u^2.$$
	Denoting the ideal of relations of $A$ by $I$ we obtain the cohomological blowup 
	$$\tilde{A}=\frac{\F[x,y,z,u,v, \xi]}{I+\xi(y,z,u,v)+(f_A(\xi))}, $$
	which has Hilbert function $H(\tilde{A})=H(A)+H(T)[1]=(1,6,6,1)$.
	Fix $\F$-bases 
	$$\tilde{A}_1=\operatorname{span}_{\F}\left\{x,y,z,u,v, \xi\right\}, \ \ \text{and} \ \ \tilde{A}_2=\operatorname{span}_{\F}\left\{u^2,uv,v^2,yv,yu,-x\xi\right\}$$
	and let $\ell\in\tilde{A}_1$ be a general linear form 
	$$\ell=ax+by+cz+du+ev+f\xi.$$
	Then the matrix for the Lefschetz map $\times\ell\colon \tilde{A}_1\rightarrow\tilde{A}_2$ and its determinant are given by
	$$M = \left(\begin{array}{cccccc}
	0 & 0 & 0 & d & 0 & -f\\
	0 & 0 & 0 & e & d & 0\\
	0 & 0 & 0 & 0 & e & 0\\
	d & e & 0 & a & b & 0\\
	0 & d & e & b & c & 0\\
	-f & 0 & 0 & 0 & 0 & -(a+\lambda f)\\
	\end{array}\right)\Rightarrow \det(M)=f^2e^4.$$ 
Thus $\ell$ is a strong Lefschetz element  for $\tilde{A}$ if and only if $e\cdot f\neq 0$.  In particular $\tilde{A}$ satisfies SLP and also WLP.
\end{ex}

\noindent Surprisingly, the analogous result to \cref{thm:SLPblowup} does not hold for the weak Lefschetz property.  
We now give an example illustrating that blowing up does not preserve WLP.
	
	\begin{exer}\label{8.7ex}
	Consider the following algebra 
\begin{eqnarray*}
A &=& \frac{\F[x,y,z,u,v]}{\operatorname{Ann}(XU^6+YU^4V^2+ZU^5V)}\\
&=& \frac{\F[x,y,z,u,v]}{\left(yz,xz,xy,vy-uz,vx,ux-vz,u^5y,u^5v^2,u^6v,u^7,v^3,x^2,y^2,z^2\right)}
\end{eqnarray*}
and its quotient  corresponding to the Thom class $\tau=u^3$
\begin{eqnarray*}
T &=& \frac{\F[x,y,z,u,v]}{\operatorname{Ann}(XU^3+YUV^2+ZU^2V)}\\
&=& \frac{\F[x,y,z,u,v]}{\left(z^2,yz,xz,y^2,xy,vy-uz,x^2,vx,ux-vz,u^2y,v^3,u^2v^2,u^3v,u^4 \right)}
\end{eqnarray*}
Consider also the cohomological blowup 
$$\tilde{A}=\frac{\F[x,y,z,u,v, \xi]}{I+\xi\cdot K+(\xi^3-u^3)}, \quad \text{where } K=\ker(A\twoheadrightarrow T).$$
\begin{enumerate}[(a)]
\item Compute the Hilbert functions of $A$ and $T$ respectively.
\item Show that both $A$ and $T$ satisfy WLP, but not SLP.
\item Show that the BUG $\tilde{A}$ does not satisfy WLP.
\end{enumerate}
	\end{exer}
	
In \cref{8.7ex} the Thom class of the map $A\to T$ has degree 3. This is the minimal possible value for such an example based on the following result.

\begin{thm}[{\cite[Theorem 8.9]{IMMSW}}]
Let $\F$ be an infinite field and let $\pi:A\to T$ be a surjective homomorphism of  graded AG $\F$-algebras such that the difference between the socle degrees of $A$ and $T$ is at most 2 and $A$ and $T$ both satisfy WLP. Then every cohomological blow-up algebra of $A$ along $\pi$ satisfies WLP.
\end{thm}

\end{document}